\documentclass[11pt,oneside,reqno]{article}

\usepackage[a4paper,width=15cm,height=24cm]{geometry}

\usepackage{amsmath,amsthm,amssymb}
\usepackage{bbm}
\usepackage{mathrsfs}

\usepackage{natbib}
\usepackage{color}
\usepackage[dvipsnames]{xcolor}
\usepackage{tikz}
\usepackage{tikz-cd}
\usetikzlibrary{shapes.geometric}
\usepackage{subcaption}
\usepackage{caption}

\usepackage{filemod}
\usepackage[breaklinks=true]{hyperref}

\theoremstyle{plain}
\newtheorem{theorem}{Theorem}[section]
\newtheorem{proposition}[theorem]{Proposition}
\newtheorem{lemma}[theorem]{Lemma}

\theoremstyle{definition}

\let\oldproof\proof
\renewcommand{\proof}{%
    \setcounter{proofclaim}{0}%
    \oldproof%
}

\makeatletter

\def\given{\typeout{Command 'given' should only be used within bracket command}}
\newcounter{@bracketlevel}
\def\@bracketfactory#1#2#3#4#5#6{
\expandafter\def\csname#1\endcsname##1{%
\addtocounter{@bracketlevel}{1}%
\global\expandafter\let\csname @middummy\alph{@bracketlevel}\endcsname\given%
\global\def\given{\mskip#5\csname#4\endcsname\vert\mskip#6}\csname#4l\endcsname#2##1\csname#4r\endcsname#3%
\global\expandafter\let\expandafter\given\csname @middummy\alph{@bracketlevel}\endcsname
\addtocounter{@bracketlevel}{-1}}%
}
\def\bracketfactory#1#2#3{%
\@bracketfactory{#1}{#2}{#3}{relax}{1mu plus 0.25mu minus 0.25mu}{0.6mu plus 0.15mu minus 0.15mu}
\@bracketfactory{b#1}{#2}{#3}{big}{1mu plus 0.25mu minus 0.25mu}{0.6mu plus 0.15mu minus 0.15mu}
\@bracketfactory{bb#1}{#2}{#3}{Big}{2.4mu plus 0.8mu minus 0.8mu}{1.8mu plus 0.6mu minus 0.6mu}
\@bracketfactory{bbb#1}{#2}{#3}{bigg}{3.2mu plus 1mu minus 1mu}{2.4mu plus 0.75mu minus 0.75mu}
\@bracketfactory{bbbb#1}{#2}{#3}{Bigg}{4mu plus 1mu minus 1mu}{3mu plus 0.75mu minus 0.75mu}
}
\bracketfactory{clc}{\lbrace}{\rbrace}
\bracketfactory{clr}{(}{)}
\bracketfactory{cls}{[}{]}
\bracketfactory{abs}{\lvert}{\rvert}
\bracketfactory{norm}{\Vert}{\Vert}
\bracketfactory{floor}{\lfloor}{\rfloor}
\bracketfactory{ceil}{\lceil}{\rceil}
\bracketfactory{angle}{\langle}{\rangle}

\newcounter{ctr}\loop\stepcounter{ctr}\edef\X{\@Alph\c@ctr}%
	\expandafter\edef\csname s\X\endcsname{\noexpand\mathscr{\X}}
	\expandafter\edef\csname c\X\endcsname{\noexpand\mathcal{\X}}
	\expandafter\edef\csname b\X\endcsname{\noexpand\boldsymbol{\X}}
	\expandafter\edef\csname I\X\endcsname{\noexpand\mathbbm{\X}}
	\expandafter\edef\csname r\X\endcsname{\noexpand\mathrm{\X}}
\ifnum\thectr<26\repeat

\newcount\minute
\newcount\hour
\newcount\hourMins
\def\now{%
\minute=\time%
\hour=\time \divide \hour by 60%
\hourMins=\hour \multiply\hourMins by 60%
\advance\minute by -\hourMins%
\zeroPadTwo{\the\hour}:\zeroPadTwo{\the\minute}%
}
\def\zeroPadTwo#1{\ifnum #1<10 0\fi#1}

\numberwithin{equation}{section}
\allowdisplaybreaks[4]

\renewcommand\section{\@startsection {section}{1}{\z@}%
{-3.5ex \@plus -1ex \@minus -.2ex}%
{1.3ex \@plus.2ex}%
{\center\small\sc\mathversion{bold}\MakeUppercase}}

\def\subsection#1{\@startsection {subsection}{2}{0pt}%
{-3.5ex \@plus -1ex \@minus -.2ex}%
{1ex \@plus.2ex}%
{\bf\mathversion{bold}}{#1}}

\def\subsubsection#1{\@startsection{subsubsection}{3}{0pt}%
{\medskipamount}%
{-10pt}%
{\normalsize\itshape}{\kern-2.2ex. #1.}}

\def\blfootnote{\xdef\@thefnmark{}\@footnotetext}

\makeatother

\renewcommand{\cite}{\citet}

\def\^#1{\ifmmode {\mathaccent"705E #1} \else {\accent94 #1} \fi}
\def\~#1{\ifmmode {\mathaccent"707E #1} \else {\accent"7E #1} \fi}

\edef\-#1{\noexpand\ifmmode {\noexpand\bar{#1}} \noexpand\else \-#1\noexpand\fi}
\def\>#1{\vec{#1}}
\def\.#1{\dot{#1}}

\def\atop{\@@atop}

\renewcommand{\leq}{\leqslant}
\renewcommand{\geq}{\geqslant}
\renewcommand{\phi}{\varphi}

\newcommand{\Var}{\mathop{\mathrm{Var}}\nolimits}

\usepackage{nomencl}
\makenomenclature

\usepackage{mwe}
\usepackage{enumitem}
\usepackage{xcolor}
\usepackage{mathtools} 

\newcommand{\R}{\mathbb{R}}
\newcommand{\N}{\mathbb{N}}
\newcommand{\E}{\mathbb{E}}

\renewcommand{\Var}{\operatorname{Var}}

\renewcommand{\P}{\mathbb{P}}

\DeclareMathOperator*{\argmin}{arg\,min}

\newtheorem*{theorem*}{Theorem}
\newtheorem*{lemma*}{Lemma}
\usepackage[normalem]{ulem}
\usepackage{authblk}
\newcommand{\Z}{\mathbb{Z}}
\DeclareMathOperator*{\argmax}{arg\,max}

\newcommand{\proofappendixheading}[1]{\bigskip\noindent\textbf{#1}\par\medskip}
\newcommand{\proofappendixstatement}[2]{\noindent\emph{#1.} #2\par\medskip}

\makeatletter
\newcommand{\MSC}[1]{%
  \begingroup
    \renewcommand\thefootnote{}
    \footnote{\textbf{2020 Mathematics Subject Classification.} #1}%
    \addtocounter{footnote}{-1}
  \endgroup
}
\newcommand{\keywords}[1]{%
  \begingroup
    \renewcommand\thefootnote{}%
    \footnote{\textbf{Key words and phrases.} #1}%
    \addtocounter{footnote}{-1}%
  \endgroup
}
\makeatother

\title{Self-Balancing Sequential Sampling: Fast Convergence with Controlled Predictability}
\author[1]{Zachary McNulty\footnote{Email: \href{mailto:zachary_mcnulty@berkeley.edu}{zachary\_mcnulty@berkeley.edu}}}
\author[2]{Daniel Raban\footnote{Email: danielraban@berkeley.edu}}
\affil[1]{Department of Mathematics, University of California, Berkeley}
\affil[2]{Department of Statistics, University of California, Berkeley}
\date{}

\begin{document}

\maketitle

\begin{abstract}
Many instances of sequential sampling, including audit and inspection scheduling, representative sampling, and treatment assignment, require selections to be distributed evenly without becoming easy to anticipate or exploit. We study a family of sequential sampling rules that adaptively bias sampling probabilities in order to achieve faster convergence of the empirical distribution to a desired target law, while keeping the resulting samples as unpredictable as possible. The resulting self-balancing sampler is simple to implement, arises naturally among a class of Markovian samplers sharing a certain invariance property, and admits a stochastic mirror-descent interpretation. Our main results show that (i) this self-balancing sampler converges at the fastest possible $O(n^{-1})$ rate with explicit dependence on biasing parameters, beating the standard $O(n^{-1/2})$ rate of IID sampling, (ii) it is the unique solution to a natural entropy-regularized optimization problem which balances the convergence rate of the empirical law and the unpredictability of the samples, and (iii) in the weak-biasing regime, the properly centered counts process converges to an Ornstein–Uhlenbeck process in the diffusive limit. Together, these results support a practical framework for reducing repeated selections and long gaps in coverage without making future selections overly predictable.
\end{abstract}

\MSC{60J10 (Primary); 62L05, 90C25, 60F17 (Secondary)}

\section{Introduction}

Consider the problem of sequentially sampling $X_i$ from $d$ outcomes $[d] := \{1,2,...,d\}$ in a way that allows the empirical distributions
$$\mu(n) \coloneqq \frac{1}{n} \sum_{i=1}^n \delta_{X_i}$$
to converge to the uniform distribution. Two naive ways of doing this are:
\begin{itemize}
    \item \emph{Uniformly at random}: Sample $X_i \overset{\operatorname{IID}}{\sim} \text{Uniform}([d])$.
    \item \emph{Deterministically/Greedily}: Sample $X_i = i \pmod d + 1$.
\end{itemize}

For these two sampling schemes, there's a trade-off between the \emph{speed} of convergence and the \emph{predictability} of the samples. Namely, the former converges at rate $O(n^{-1/2})$ in the total variation distance, while the latter converges at the much faster rate $O(n^{-1})$. On the other hand, the former process is as unpredictable as possible, maximizing the entropy at each sampling step, while the latter is perfectly predictable.

This tradeoff becomes readily apparent in the problem of scheduling audits or inspections, such as testing for doping in professional sports. A sports regulatory body is tasked with routinely testing teams for the use of performance enhancing drugs, but any deterministic schedule has the potential to be taken advantage of, since the players can temporarily cease doping in anticipation of the tests. On the other hand, sampling teams IID to receive doping tests can lead to wasteful behavior, such as spending resources to test a team multiple times in relatively short succession, while letting other teams remain temporarily untested; in effect, temporary imbalances in the empirical distribution of audited teams reflect inefficiencies in the order of teams chosen. Moreover, empirical evidence suggests \textit{biasing} inspections towards individuals with certain risk factors can improve outcomes \cite{antidoping}. Similar tradeoffs arise in the scheduling of restaurant health inspections, where inspection agencies often seek to prioritize higher-risk establishments or those with poor inspection histories while still preserving enough unpredictability to prevent strategic compliance behavior \cite{Buchholz2002RiskBasedRestaurant}.

Another closely related problem is how to sample from a population while keeping the realized sample balanced across several overlapping strata. For example, suppose we are constructing a citizen's assembly of 100 members out of a volunteer pool of 1000, and we would like our assembly to be demographically representative across multiple overlapping characteristics, such as gender, age, and education level. We may attempt to balance these characteristics using a deterministic or quota-constrained selection algorithm for choosing a committee, but, as noted in \cite{flanigan2021fair}, some algorithms can lead to somewhat undemocratic outcomes in which certain members of the volunteer pool have a minuscule chance of being chosen for the committee. On the other hand, a simple random sample can result in a committee which is not properly representative for each characteristic. 

Lastly, in randomized controlled trials there is inherent tension between ensuring covariates known to influence treatment effects are balanced across groups, while maintaining enough randomness to provide robustness against unknown prognostic factors. \cite{Harshaw01102024} explore this trade-off between covariate balance and robustness, but their sampling strategy is inherently non-sequential. On the other hand, the classical work of \cite{PocockSimon1975Sequential} provides a sequential analogue, biasing each assignment toward the treatment group that would reduce observed covariate imbalance. While effective as a balancing heuristic, the resulting procedure is not derived from optimizing any explicit objective, making the balance/robustness tradeoff relatively opaque.

In this paper, we aim to explore the trade-off between these two aims---balance/convergence and robustness/unpredictability---for a natural family of sequential sampling rules. Heuristically, we show that by adaptively biasing our sampling process, we can limit predictability while maintaining fast convergence rates. For this purpose, consider the following sequential sampling rule, informally described as ``downweight the outcome's probability and renormalize'': Fix some $\alpha \in (0,1)$, and start with some initial distribution $p(0) \in \Delta_+^d$, where 
    \[\Delta_+^d := \{x \in \R^d : x_i > 0, \sum x_i = 1\}\]
\noindent is the d-dimensional unit simplex with full support. At step $n$, recursively do the following:
\begin{enumerate}
    \item[] Step 1: Sample an outcome $X_n \sim p(n-1)$.
    \item[] Step 2: Define $p(n)$ by downweighting $p_{X_n}(n-1)$ by a factor of $\alpha$ and renormalizing the probability vector, i.e.
    $$p_{i}(n) \coloneqq \begin{cases}
            \frac{\alpha \cdot p_{i}(n-1)}{1- (1-\alpha) \cdot p_{X_n}(n-1)} & \text{if $i = X_n$}, \\[5pt]
            \frac{p_{i}(n-1)}{1 - (1-\alpha) \cdot p_{X_n}(n-1)} & \text{if $i \neq X_n$}.
        \end{cases}$$
\end{enumerate}
See Figure~\ref{fig:uniform_downweight_normalize} for a visual schematic of the update rule.

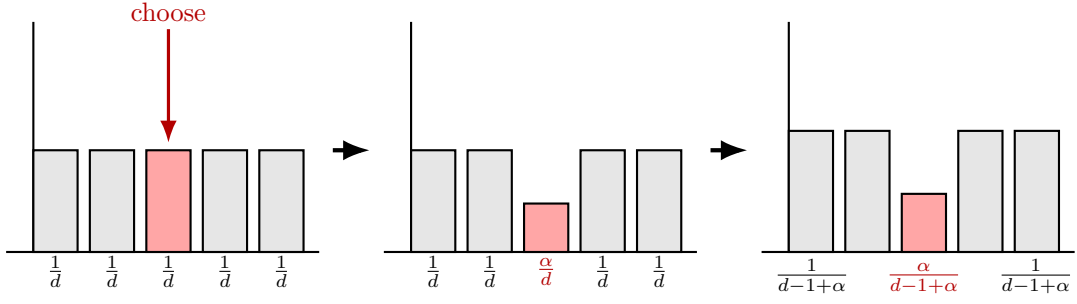
\begin{figure}
\centering
\begin{tikzpicture}[
    >=Latex,
    font=\small,
    x=0.9cm,
    y=3.2cm,
    bar/.style={draw, thick, fill=gray!20},
    chosen/.style={draw, thick, fill=red!35}
]

\def\d{5}            
\def\k{3}            
\def\w{0.65}         
\def\gap{0.18}       
\def\arrowlen{0.55}  

\def\xA{0}
\def\xB{5.55}
\def\xC{11.10}

\def\hunif{0.42}
\def\hdown{0.20}
\def\hnormother{0.50}
\def\hnormchosen{0.24}

\def\xleft{-0.4}
\def\xright{4.2}

\pgfmathsetmacro{\arrowABstart}{((\xA+\xright)+(\xB+\xleft)-\arrowlen)/2}
\pgfmathsetmacro{\arrowABend}{\arrowABstart+\arrowlen}
\pgfmathsetmacro{\arrowBCstart}{((\xB+\xright)+(\xC+\xleft)-\arrowlen)/2}
\pgfmathsetmacro{\arrowBCend}{\arrowBCstart+\arrowlen}

\draw[thick] (\xA+\xleft,0) -- (\xA+\xright,0);
\draw[thick] (\xA,0) -- (\xA,0.95);

\foreach \i in {1,...,5}{
    \pgfmathsetmacro{\x}{\xA + (\i-1)*(\w+\gap)}
    \ifnum\i=\k
    \draw[chosen] (\x,0) rectangle ++(\w,\hunif);
\else
    \draw[bar] (\x,0) rectangle ++(\w,\hunif);
\fi
    \node at (\x+0.5*\w,-0.08) {\(\frac{1}{d}\)};
}

\pgfmathsetmacro{\xchosenA}{\xA + (\k-1)*(\w+\gap) + 0.5*\w}
\draw[->, very thick, red!70!black] (\xchosenA,0.92) -- (\xchosenA,\hunif+0.03);
\node[red!70!black] at (\xchosenA,0.99) {choose};

\draw[->, ultra thick] (\arrowABstart,0.42) -- (\arrowABend,0.42);

\draw[thick] (\xB+\xleft,0) -- (\xB+\xright,0);
\draw[thick] (\xB,0) -- (\xB,0.95);

\foreach \i in {1,...,5}{
    \pgfmathsetmacro{\x}{\xB + (\i-1)*(\w+\gap)}
    \ifnum\i=\k
        \draw[chosen] (\x,0) rectangle ++(\w,\hdown);
        \node[red!70!black] at (\x+0.5*\w,-0.08) {\(\frac{\alpha}{d}\)};
    \else
        \draw[bar] (\x,0) rectangle ++(\w,\hunif);
        \node at (\x+0.5*\w,-0.08) {\(\frac{1}{d}\)};
    \fi
}

\draw[->, ultra thick] (\arrowBCstart,0.42) -- (\arrowBCend,0.42);

\draw[thick] (\xC+\xleft,0) -- (\xC+\xright,0);
\draw[thick] (\xC,0) -- (\xC,0.95);

\foreach \i in {1,...,5}{
    \pgfmathsetmacro{\x}{\xC + (\i-1)*(\w+\gap)}
    \ifnum\i=\k
        \draw[chosen] (\x,0) rectangle ++(\w,\hnormchosen);
    \else
        \draw[bar] (\x,0) rectangle ++(\w,\hnormother);
    \fi
}

\pgfmathsetmacro{\xchosenC}{\xC + (\k-1)*(\w+\gap) + 0.5*\w}
\node[red!70!black] at (\xchosenC,-0.1075) {\( \frac{\alpha}{d-1+\alpha}\)};

\pgfmathsetmacro{\xleftlab}{\xC + \w + 0.5*\gap}
\pgfmathsetmacro{\xrightlab}{\xC + 4*\w + 3.5*\gap}
\node at (\xleftlab-0.4,-0.10) {\( \frac{1}{d-1+\alpha}\)};
\node at (\xrightlab+0.4,-0.10) {\( \frac{1}{d-1+\alpha}\)};

\end{tikzpicture}
\caption{Schematic of the self-balancing update starting from a uniform distribution: choose an outcome, downweight it by \(\alpha\), then renormalize.}
\label{fig:uniform_downweight_normalize}
\end{figure}

This sequential sampling rule is still random, and intuitively, it generates empirical distributions which are closer to uniform because it de-emphasizes outcomes which are currently over-represented in the empirical distribution. More precisely, if we denote the counts process
    \[N_i = N_i(n) \coloneqq \# \{ 1 \leq j \leq n : X_j = i\},\] 
then this rule generates samples according to
$$p_i = p_{i}(n) \coloneqq \frac{p_{i}(0) \cdot \alpha^{N_i(n)}}{\sum_{j=1}^d p_{j}(0) \cdot \alpha^{N_j(n)} }.$$

\noindent Hence we can see that this sampler interpolates the two naive sampling methods: when $\alpha = 1$ (and $p(0)$ is uniform), we recover the uniform sampler, and when $\alpha \downarrow 0$, we recover greedy least-count sampling with uniformly random tie-breaking; if ties are broken deterministically, we recover deterministic sampling. 

As we will later see, the same scheme works for non-uniform target distributions by replacing a uniform $\alpha$ by a suitable $\alpha_i$ for each $i$; more precisely, for a target distribution $p^*$, one can set $\alpha_i = e^{-c/p_i^*}$ for every $i$, where the choice of $c > 0$ determines the trade-off between balance and predictability. Perhaps most unexpectedly, this simple sampling rule admits three complementary interpretations: as the essentially unique invariant potential-based sampler (Proposition~\ref{prop:charactizing_v_invariance}), as the Gibbs optimizer of an entropy-regularized count-imbalance objective (Theorem~\ref{thm:instantaneous_optimization}), and as an entropic stochastic mirror-descent algorithm (Theorem~\ref{thm:stochastic_mirror_descent}).

\paragraph*{Related Work.}
Our sampling rule is related to several bodies of work on negatively dependent and history-dependent sampling. Negatively reinforced and de-preferential urn models assign lower selection probabilities to colors that are currently overrepresented, typically by applying a fixed decreasing weight function to the vector of urn proportions \citep{Bandyopadhyay_Kaur_2018,kaur2019negatively}. These works establish laws of large numbers and central limit theorems under general replacement mechanisms. Our model differs in that observations directly increment the selection counts and the sampling probabilities depend exponentially on the raw counts, $p_i(n) \propto p_i(0) e^{-(-\log \alpha_i)N_i(n)}$, rather than sample proportions. Equivalently, the weight applied to the empirical proportion $\mu_i(n)$ is $e^{-n(-\log \alpha_i)\mu_i(n)}$, so for fixed $\alpha_i$, the strength of the repulsion on the empirical measure scale grows with time. This produces a positive-recurrent centered-count process and empirical convergence on the $n^{-1}$ scale, rather than ordinary central-limit fluctuations.

A second closely related line of work concerns self-interacting and self-repellent random walks. Recent self-repellent random-walk constructions modify a base Markov kernel by downweighting states having large empirical occupation proportions and establish almost-sure convergence and central limit theorems with reduced asymptotic variance \cite{doshi2024selfrepellentrandomwalksgeneral}. Exponentially interacting reinforced walks on complete graphs have also been studied, particularly for systems of several attracting or repelling walkers \cite{ROSALES2022353}. In contrast, our sampler has no graph constraint or dependence on the current location, and its exponential weights are applied to cumulative counts rather than to occupation proportions with a fixed interaction strength. As above, a time-inhomogeneous choice of $\alpha$ parameters expresses the process in terms of proportions, rather than counts, connecting the two regimes.

The trade-off between balance and unpredictability has a long history in restricted randomization
for sequential experiments. Beginning with Efron's biased-coin design and its subsequent extensions, these procedures
increase the probability of assigning the next subject to an underrepresented treatment while retaining sufficient
randomness to limit selection bias \cite{05287b01-9a36-3190-b702-371959e17200,7708aeae-3258-37ff-b421-c7c45f9f9dd0,https://doi.org/10.1111/1467-985X.00564,https://doi.org/10.1111/j.1467-9876.2004.00436.x,10.1214/09-AOS758}.
This literature commonly evaluates a design through both its resulting treatment imbalance and its \textit{selection bias},
the probability with which an observer can predict the next assignment. Within this framework, our sampler uses assignment probabilities that depend
 exponentially on the weighted cumulative counts. The resulting strength of correction increases with the observed
 count imbalance, rather than remaining at a fixed bias level. We further derive this exponential rule from an explicit
  entropy-regularized count-imbalance objective and analyze its convergence, centered-count dynamics, and weak-bias
  diffusion behavior. In this sense, the self-balancing sampler may be viewed as a multidimensional family of
   biased-coin designs for which the balance--unpredictability trade-off is both explicitly parameterized and
   amenable to a unified stochastic-process analysis.

Most closely related in the biased-coin literature is the work of \cite{Hu_2016}, which
provides a generalized framework for extending these classical ideas to an arbitrary number of treatment groups and positive rational target allocation proportions. Hu gives a Foster–Lyapunov drift condition under which the empirical allocation converges to the target at rate $O_p(n^{-1})$. For fixed bias parameters, our self-balancing sampler belongs to this general framework; in particular, Hu’s theorem recovers the qualitative $O_p(n^{-1})$ behavior in aligned rational cases where its drift condition applies. Our convergence results in Section~\ref{sec:convergence_rate} strengthen and extend this conclusion for our sampler in a number of ways: (i) we prove an $L^1$ total-variation bound of order $n^{-1}$, (ii) our convergence bound has explicit dependence on the bias parameters and initialization, (iii) our results do not have a drift bound as a hypothesis, (iv) our results apply for arbitrary full-support initial laws and arbitrary positive target distributions, and (v) we prove a pathwise $\log n/n$ bound, as well. Moreover, we also consider time-inhomogeneous biasing, which does not fit into Hu's framework. Beyond these convergence results, we derive the exponential sampling rule from an entropy-regularized imbalance objective, interpret it as stochastic mirror descent, and study its weak-bias diffusion limit.

\paragraph*{Outline.}
Section \ref{sec:potential_biased_sampling} shows that such samplers arise naturally by enforcing an invariance property on the sampling probabilities $p(n)$, namely that they depend only on \emph{relative} differences in the counts $N_i$. We show that after properly centering the counts process, it descends to an irreducible and positive recurrent Markov chain on the centered counts, a process which exhibits strong mean-reversion properties.

Then, in Section \ref{sec:convergence_rate} we show our first main result: that for any choice of $\alpha$, the self-balancing sampler converges at rate $O(1/n)$ in expectation, matching the fastest possible rate up to constants, while it pathwise remains within $O(\log n / n)$ of its target law almost surely, which is optimal up to a logarithmic factor. Hence, the self-balancing property of this sampler drastically improves upon the convergence rate of IID sampling, and matches the convergence rate of the naive greedy sampler. Moreover, we quantify the influence the bias parameter $\alpha$ has on the leading order constant and show there are two regimes: the leading constant is order $(\log(1/\alpha))^{-1/2}$ when $p(0) = p^*$, where $p^*$ is the asymptotic sampling distribution, and is order $(\log(1/\alpha))^{-1}$ otherwise. We additionally show that these two regimes apply for time-inhomogeneous biasing parameters, as well. A key technical ingredient, Lemma~\ref{lem:drift_lemma} (which is interesting in its own right), is a quantitative specialization of the drift criterion of \cite{PemantleRosenthal1999}, which provides the explicit parameter dependence needed in our convergence bounds.

In Section \ref{sec:optimization_problem}, we show that the self-balancing sampler admits a natural interpretation through an entropy-regularized optimization problem, balancing minimization of the mean-squared imbalance in the counts with maximizing entropy. This formalizes the balance/convergence versus robustness/unpredictability trade off in terms of an explicit objective, a limitation of the classical procedure \cite{PocockSimon1975Sequential}. We further show that the self-balancing sampler implements stochastic mirror descent on the space of sampling distributions. Section \ref{appendix:optimal_control} later shows that the sampler further fits within the broader framework of entropy-regularized optimal control.

In an effort to further quantify the trade-off between convergence and unpredictability, we develop diffusion
 limits in Section~\ref{sec:diffusion_approximations} and an adversarial notion of predictability in Section
~\ref{sec:adversarial_interpretation}. The latter defines the unpredictability as the error-rate an adversary
   trying to predict the next sampled index would achieve if they behaved optimally, which is possibly more
   natural than the entropy-based characterization of Section~\ref{sec:optimization_problem}. We are
   particularly interested in a \emph{weak-biasing regime}, in which $\alpha \uparrow 1$, which is the
   natural regime under which diffusion limits arise. Using these tools, we provide some non-rigorous analysis
   into how the behavior of the sampler at stationarity changes as we vary the bias parameter $\alpha$.
In this weak-biasing regime, the diffusion heuristics approximate the centered counts by an Ornstein-Uhlenbeck process with stationary
fluctuations of order $(\log(1/\alpha))^{-1/2}$, matching the upper-bounds derived in Section~\ref{sec:convergence_rate} and
suggesting how stronger biasing enforces mean-reversion while reducing randomness.

Complementing this, the adversarial analysis in Section~\ref{sec:adversarial_interpretation}
translates these fluctuations into prediction accuracy. In the $d=2$ case with uniform target law $p^*$,
we show, through explicit analysis of the stationary law of the centered counts process, that an adversary
can predict the next sample with accuracy $\frac{1}{2} + \Theta(\sqrt{\log(1/\alpha)})$ in the weak-bias regime
$\alpha \uparrow 1$. In more general settings, through non-rigorous analysis we provide evidence of the following conjecture:
 If $p^*$ does not have a unique maximizer, the adversary can exploit fluctuations
among the tied outcomes, and its excess prediction accuracy above the naive baseline $\max_i p_i^*$ is again of order
$\sqrt{\log(1/\alpha)}$ as in the $d=2$ case.

Empirical demonstrations and applications to the citizen's assembly model of our results are
provided in Appendix~\ref{appendix:empirical_results}. For clarity of exposition, proofs are
postponed until Appendix~\ref{appendix:proofs}. 

\section{Basic results}

\subsection{Definitions}

A \textit{sequential sampling rule} on $[d]$ consists of a filtered probability space $(\Omega, \cF, (\cF_n)_{n \geq 0}, \P)$, an $\cF_n$-adapted sequence $(\rho_n)_{n \geq 0}$ of random distributions on $[d]$, and random variables $(X_n)_{n \geq 1}$ on $[d]$ so that
    \[\P(X_{n+1} = i \mid \cF_n) = \rho_n(i), \quad \forall i \in [d]. \]
\noindent Most commonly, we have $\cF_n = \sigma(X_1,..., X_n)$. Such a sequential sampling rule selects its distribution $\rho_n$ based solely on the previously observed outcomes, and then conditional on $\rho_n$ generates the next outcome $X_{n+1}$ by sampling from $\rho_n$. 

We say such a rule is \textit{Markovian} if it selects its next sample $X_{n+1} \in [d]$ based solely on the current selection counts $\{N_i(n)\}_{i=1}^d$ where:
\[N_i(n) := |\{ 1 \leq j \leq n : X_j = i\}|.\]
\noindent  Formally, such a sampling rule is given by a measurable function $\pi:\N^d \to \Delta^d$ where $\Delta^d$ is the unit simplex in $\R^d$, and samples are generated via:
\[\P\left(X_{n+1} = i \mid \{X_j\}_{j \leq n}\right) = \pi_i(N(n)).\]
\noindent The self-balancing sampler is then a Markovian sampling rule with sampling function $\pi_i = p_i$. In this paper, we will show that this sampler (and more generally samplers with $\alpha_i \in (0,1)$ possibly different for each outcome) arises naturally within this class of sampling rules after enforcing a specific symmetry on the sampling function $\pi$, and that it exhibits nice optimality properties. Note we can express this sampling rule as a softmax
\begin{equation}
    p_i = p_i(n) = \frac{p_{i}(0) \cdot e^{-\beta_i N_{i}(n)}}{\sum_{j=1}^d p_{j}(0) \cdot e^{-\beta_j N_{j}(n) }},
    \label{eqn:self_balancing_sampler_transition_probs}
\end{equation}

where $\beta_i = - \log \alpha_i$. As we will see below, these $\beta_i$ are in some sense a more natural parameterization. Heuristically, they can be interpreted as follows: for $i,j \in [d]$, if $\beta_i^{-1} /\beta_j^{-1} = \frac{p}{q}$ is rational, then $p$ samples of $i$ are weighted the same as $q$ samples of $j$ under this sampling scheme. Thus $\beta_i^{-1}$ measures the relative weight assigned to sampling $i \in [d]$. Higher weighted indices tend to be selected more often.   

In this paper, we are interested in how the empirical sampling distributions
\[\mu(n) := \frac{1}{n} \sum_{i=1}^n \delta_{X_i}\]
\noindent evolve over time under the self-balancing sampler. We show the limiting distribution is
\[p_i^* := \frac{\beta_i^{-1}}{\sum_{j=1}^d \beta_j^{-1}}.\]
\noindent Namely, asymptotically we just select each $i$ proportional to its relative weight $\beta_i^{-1}.$

Since the counts $N_i(n)$ drift off to infinity over time, we will require some form of centering to properly study this process. There are two natural centering schemes. The first, which we call the \textit{mean-centered counts}, just centers each count by its expected limiting proportion:
\[Y_i(n) := N_i(n) - n p_i^*.\]
The second, which we call the \textit{projection-centered counts}, reweights the counts according to their relative weights $\beta_i^{-1}$, fixes the coordinate $d$ as a baseline, and compares the counts against this baseline:
\[\widetilde{Y}_i(n) := \beta_iN_i(n) - \beta_d N_d(n) = \beta_i\beta_d \cdot  [\beta_d^{-1} N_i(n) - \beta_i^{-1} N_d(n)].\]
\noindent We can express the sample probabilities $p_i$ directly in terms of these centered counts via
\[p_i = \frac{p_{i}(0) \cdot e^{- \beta_i Y_{i}(n)}}{\sum_{j=1}^d p_{j}(0) \cdot e^{- \beta_j Y_{j}(n) }}= \frac{p_{i}(0) \cdot e^{- \widetilde{Y}_{i}(n)}}{\sum_{j=1}^d p_{j}(0) \cdot e^{-\widetilde{Y}_{j}(n) }}.\]

Lastly, we will often be interested in the \textit{weak-biasing regime}, where all the bias parameters tend to zero. To avoid changing the target proportions $p^*$, we want the ratios $\beta_i/\beta_j$ to remain constant. Let $\beta \downarrow 0$ denote any sequence of bias $\beta^{(k)}$ of the form $\beta^{(k)}_i = c_k \beta_i$ for $c_k \downarrow 0$ and $\beta_i$ fixed. Similarly, define $\beta \uparrow \infty$ if $c_k \uparrow \infty$. 

\subsection{Potential-based sampling}
\label{sec:potential_biased_sampling}

In this section, we show that the self-balancing sampler arises naturally by enforcing a certain invariance property on our class of Markovian sampling rules, and that under 
this relation the counts process $(N(n))_{n \in \N}$ descends 
to our projection-centered counts $(\tilde{Y}(n))_{n \in \N}$. This projected chain is a Markov chain in its own right, and under mild assumptions is irreducible, positive recurrent, and admits a unique stationary distribution. Proofs can be found in Appendix~\ref{appendix:2_2_proofs_potential_based_sampling}.

 For a Markovian sampling rule, one natural choice for the sampling function $\pi$ is to define a potential $\phi_i: \N \to (0, \infty)$ for each coordinate and to select a sample proportional to its potential
\[\pi^\phi_i(x) := \frac{\phi_i(x_i)}{\sum_{j=1}^d \phi_j(x_j)}, \quad \forall x \in \N^d.\]
\noindent We say a Markovian sampling rule $\pi$ is \textit{translation invariant} if
\[\pi(x + \mathbf{1}) = \pi(x), \quad \forall x \in \N^d.\]
\noindent Translation-invariance expresses the opinion that our next sample $X_{n+1}$ should depend only on the relative differences in the sampling counts $\{N_i(n)\}_{i=1}^d$ rather than their absolute sizes. More generally, we say $\pi$ is \textit{$v$-invariant} for some $v \in \N^d$ if
\[\pi(x + v) = \pi(x), \quad \forall x \in \N^d.\]
\noindent Here $v$-invariance expresses the heuristic that sampling $i$ a total of $v_i$ times ought to be equivalent to sampling $j$ a total of $v_j$ times. Thus $(v_i)$ plays an analogous role to the relative weights $(\beta_i^{-1})$. 

Note that we can view v-invariance as inducing an equivalence relation $\sim_v$ on $\N^d$:
    \[x \sim_v y \iff x = y + k v, \quad k \in \Z.\]
\noindent Under this relation and a slight technical assumption, our counts process $(N(n))_{n \in \N}$ descends to a Markov chain on $\Z^{d-1}$. 

\begin{proposition}[Projection-Centered Chain]
    If $\pi$ is $v$-invariant and $\mathrm{gcd}(v_1,..., v_d) = 1$ 
    then $v_dN_i(n) - v_i N_d(n),$ is a Markov chain on $\Z^{d-1}$.
     \label{prop:projection_centered_counts_are_markovian}
\end{proposition}

\noindent Under the assumptions of the above proposition, the relation $\sim_v$ is equivalent to
    \[x \sim_v y \iff v_dx_i - v_ix_d = v_dy_i - v_iy_d.\]

\noindent Note that
\[v_dN_i(n) - v_i N_d(n) = \beta_d^{-1}N_i(n) - \beta_i^{-1}N_d(n),\]
\noindent which, up to constants, is just our projection-centered counts $\tilde{Y}(n)$. Alternatively, we can express the above equivalence as
\[\frac{x_i}{v_i} - \frac{1}{d} \sum_{j=1}^d \frac{x_j}{v_j} = \frac{y_i}{v_i} - \frac{1}{d} \sum_{j=1}^d \frac{y_j}{v_j}, \quad \forall i \in [d],\]
which shows the projection-centered counts $\tilde{Y}(n)$ center the reweighted counts $\beta_iN_i(n)$ by their mean, rather than the long-term
average $np^*$ as $Y(n)$ does. Now, we show the self-balancing sampler is essentially the only $v$-invariant sampling rule generated by a potential.

\begin{proposition}
    If  $\pi^\phi$ is $v$-invariant, there exists $b, c_{i,k} > 0$ so
    \[\phi_i(mv_i + k) = c_{i,k} b^m, \quad \forall m \in \N, k \in \{0, 1, \ldots, v_i-1\}.\]
    \noindent Moreover, if $\phi_i$ is either log-convex or log-concave, then $\phi_i(n) = c_i \alpha_i^n$ where
    \[\log(\alpha_i) = \log(b) v_i^{-1}.\]

    \label{prop:charactizing_v_invariance}
\end{proposition}

\noindent In the case of translation-invariance, the log-convexity/concavity condition is unnecessary and we recover the self-balancing sampler in the case of uniform $\beta_i \equiv \beta$. Note that the sampling function for the self-balancing sampler is $v$-invariant only in the case the ratios $\beta_i^{-1}/\beta_j^{-1}$ are all rational, enforcing a relatively minor constraint on $\beta$. 

    When we have potential $\phi_i(n) = c_i \alpha_i^n$ so that $\pi^\phi$ is v-invariant, it is easy to see $\pi^\phi$ is invariant with respect to
        \[v' := v / \mathrm{gcd(v_1,..., v_d),}\]
    \noindent as well. Hence in this setting of interest in this paper, we may always assume $\mathrm{gcd}(v) = 1$ so by Proposition~\ref{prop:projection_centered_counts_are_markovian} our projection-centered counts are Markovian.

In a sense, the projection-centered counts are more natural because they live 
on a countable space $\Z^{d-1}$. As the self-balancing sampler has strong mean-reversion 
properties, it is natural to expect the projection-centered counts are positive recurrent, and hence
have some stationary distribution. This is indeed the case.

\begin{lemma}
    For $\beta_i=cv_i^{-1}$, where $v\in\N_{>0}^d$ and $c>0$, the projected chain
    \[
        \left(v_dN_i(n)-v_iN_d(n)\right)_{i=1}^{d-1}
    \]
    \noindent with potential $\phi_i(n) = p_i(0)e^{-\beta_i n}$ as in \eqref{eqn:self_balancing_sampler_transition_probs}
    is irreducible and positive recurrent, and thus has a unique stationary
    distribution $\eta$ on
    \[
        \left\{\left(v_d n_i-v_i n_d\right)_{i=1}^{d-1}:n\in\Z^d\right\}
        \subseteq\Z^{d-1}.
    \]
    \label{lem:projection_centered_has_stationarity}
\end{lemma}
In general, determining the stationary distribution $\eta$ is challenging, as the chain is non-reversible for $d > 2$.
However, in the special case $d=2$, we can calculate $\eta$ explicitly, which we will do in Lemma \ref{lem:2d_stationary_project_centered} below. 
Positive recurrence is just the first manifestation of the sampler's mean-reversion properties, which we will explore more in the following sections. 

Unsurprisingly, at stationarity the  sampler generates samples according to the target $p^*$.

\begin{lemma}
Suppose $\pi$ is $v$-invariant and the projection-centered chain $\tilde{Y}^v(n) = v_dN_i(n) - v_i N_d(n)$ has stationary distribution $\eta$ on $\Z^{d-1}$.
 If we sample $y \sim \eta$ and $X \mid y \sim \pi(y)$, then:
    \[\P(X = i) = \frac{v_i}{\sum_j v_j} = p_i^*.\]
 \label{lem:sampling_at_stationarity_gives_pstar}
\end{lemma}

\subsection{Fast convergence of the empirical distribution for self-balancing sampling}

\label{sec:convergence_rate}

In this section, we show that for any choice of $\alpha_i \in (0,1)$, the self-balancing sampler converges to the desired limiting distribution at a rate faster than IID sampling, and in fact converges as fast as possible. Below, for functions $f,g:\N \to [0, \infty)$ we say $f = O(g)$ if $\lim\sup_{n\to \infty} f/g < \infty$. Further, we say $f = \Theta(g)$ if $f = O(g)$ and $g = O(f)$. Proofs can be found in Appendix~\ref{appendix:2_3_convergence_rate_proofs}.

As a preliminary result, we note that the empirical distributions associated to many sequential sampling rules generally converge, which is a consequence of their close relationship with the C\'esaro means of the 
sampling distributions $p(n)$. However, this relationship is at the scale of $O(1/\sqrt n)$, which is the best one can hope for in general.

\begin{proposition}\label{prop:empirical_averages}
    $$\mathbb E \left[\max_{1 \leq i  \leq d} \left |\mu_{i}(n) - \frac{1}{n} \sum_{m=0}^{n-1} p_{i}(m) \right| \right]\lesssim \frac{d}{\sqrt n}.$$
    Moreover, if the projection-centered count process for a sequential sampling rule is irreducible and positive recurrent, then both of these quantities converge almost surely.
\end{proposition}

For any empirical distribution sampled on $d \geq 2$ values (each with nonzero probability), the fastest possible big-O convergence rate is $O(1/n)$, because for any distribution $q \in \Delta^d$, there must always be a subsequence along which $\|\mu(n) - q\|_{\operatorname{TV}} \geq C/n $ for some constant $C$. Indeed,
$$\| \mu(n) - q\|_{\operatorname{TV}} = \frac{1}{2} \sum_{i=1}^d \left| \frac{N_i}{n} - q_i \right| \geq \frac{1}{2n}\sum_{i=1}^d \min_{z \in \Z} |z - nq_i|.$$
If $q_i$ is irrational for some $i$, then $\{n q_i \pmod 1\}_{n=1}^\infty$ is dense in $[0,1]$, so there are infinitely many $n$ such that $\min_{z \in \Z} |z - nq_i| \geq 1/3$, giving an overall lower bound of $\|  \mu_n - q\|_{\operatorname{TV}} \geq \frac{1}{6n}$ for infinitely many $n$. On the other hand, if all $q_i$ are rational with least common denominator $m$, then for any $n$ not divisible by $m$, $\min_{z \in \Z} |z - nq_i| \geq 1/m$ for some $i$. Thus, for infinitely many $n$, we get $\| \mu_n - q\|_{\operatorname{TV}} \geq \frac{1}{2mn}$.

IID sampling converges at the rate $\Theta(1/\sqrt n)$ in expectation, as sampling from $q$ gives
$$\E[\| \mu_n - q\|_{\operatorname{TV}}] = \frac{1}{2\sqrt n} \sum_{i=1}^d \E \left[\left|\frac{N_i - nq_i}{\sqrt n}\right|\right],$$
where each expectation term is $\Theta(1)$ by the central limit theorem. However, for our self-balancing sampler, the
empirical distribution converges much faster than IID sampling. Theorem~\ref{thm:average_case_conv_rate} shows that
 in expectation, the convergence actually matches the fastest possible rate of $O(1/n)$, regardless of the choice
 of $\alpha_i \in (0,1)$.

 Such a rate commonly appears in such biased sampling procedures, and by specializing the classical
 Foster-Lyapunov drift condition (see e.g. \cite{meyn2009markov}) to the sequential sampling setting of interest, \cite{Hu_2016}
 provides a sufficient drift condition for this to occur for a broad family of sequential sampling rules. However, for our sampler, by developing a quantitative version of the classical Foster-Lyapunov
 drift conditions (Lemma~\ref{lem:drift_lemma}), we are able to get more explicit control on the leading order constants.

\begin{theorem}[Average-case convergence rate] \label{thm:average_case_conv_rate}
    For the self-balancing sampling rule, the empirical distribution $\mu(n)$ converges at a rate of $1/n$:
    $$\E\| \mu(n) - p^* \|_{\operatorname{TV}} = O(1/n).$$
    More granularly, for $\max_i \beta_i \leq 1$,
    $$\mathbb E \| \mu(n) - p^*\|_{\operatorname{TV}} \lesssim \begin{cases}
        \frac{1}{n}\sum_i \frac{1}{\beta_i} & \text{if $p(0) \neq p^*$,} \\
        \frac{1}{n}\sum_i \frac{1}{ \sqrt{\beta_i}} & \text{if $p(0) = p^*$.}
    \end{cases}$$
    where the constant depends only on $p(0)$, $p^*$, $\kappa := \frac{\max_i \beta_i}{\min_i \beta_i}$, and $d$.
\end{theorem}

As one might expect, increasing the bias $\beta$ tends to decrease the gap between the target sampling law $p^*$ and the empirical sampling distribution. To see why these two regimes arise, note when $p(0) = p^*$, the sampling law when $Y(0) = 0$ is already $p^*$. Hence it only has to correct the fluctuations about this stationary sampling law. When $p(0) \neq p^*$, the sampling law is $p^*$ only when $Y(n) = c_i/  \beta_i$, so the sampler must additionally correct this initial gap of order $\beta_i^{-1}$ caused by this mismatch.

We can analyze time-inhomogeneous biasings $\beta =\beta(n)$ (and the corresponding $\alpha = \alpha(n)$), as well; scaling $\beta$ by a constant leaves $p^*$ unchanged, so we may study the weak-biasing parameter regime of $\beta_i(n) = a(n)b_i$ for all $i$, where the $b_i$ are fixed and $a(n) \downarrow 0$ as $n \to \infty$. In this case, the sampler becomes $p_i(n) \propto p_i(0) e^{-a(n) b_i N_i}$.

\begin{theorem}[Convergence rate for time-inhomogeneous biasing] \label{thm:time-dep-conv-rate}
    If $\beta_i(n) = a(n)b_i$ for all $i$, where $a(n) \downarrow 0$ as $n \to \infty$ and $\sup_n a(n+1)^{-1} - a(n)^{-1} < \infty$, then
    $$\mathbb E \| \mu(n) - p^*\|_{\operatorname{TV}} \lesssim \begin{cases}
        \frac{1}{n} \sum_i \frac{1}{a(n)} & \text{if $p(0) \neq p^*$,} \\
        \frac{1}{n} \sum_i \frac{1}{\sqrt{a(n)}} & \text{if $p(0) = p^*$,}
    \end{cases}$$
    where the constant depends only on $p(0),d,\sup_n a(n+1)^{-1} - a(n)^{-1}$, and the $b_i$.
\end{theorem}

For example, if we take $p(0) = p^*$ and $\beta(n) = n^{-\gamma} b_i$ for $0 < \gamma \leq 1$ (so that $\alpha_i(n) = e^{-n^{-\gamma}b_i} \approx 1 - n^{-\gamma}b_i$), we get a sampler with convergence behavior that interpolates between the rates of iid sampling and deterministic sampling:
    $$\mathbb E \| \mu(n) - p^*\|_{\operatorname{TV}} \lesssim n^{-1+\gamma/2}.$$
    The scaling $\beta_i(n) = b_i/n$ is also the regime in which the transition probabilities become fixed functions of the empirical occupation measure $p_i(n) \propto p_i(0) e^{-b_i\mu_i(n)}$, bringing the model into the stochastic-approximation scale commonly considered for negatively reinforced urns and self-repellent nonlinear Markov chains.

The following shows that even pathwise, the self-balancing sampler still improves upon the $O(n^{-1/2})$ convergence rate of IID sampling, staying within distance $O(\log n/n)$ of its target $p^*$ almost surely.

\begin{theorem}[Pathwise convergence rate] \label{thm:limit_probabilities}
    For the self-balancing sampling rule, the empirical distribution
    $$\| \mu(n) - p^*\|_{\operatorname{TV}} \xrightarrow{a.s.} 0$$
    as $n \to \infty$. Moreover, almost surely,
    $$\| \mu(n) - p^*\|_{\operatorname{TV}} \lesssim \frac{\log n}{n} \sum_{i=1}^d \frac{1}{\beta_i}$$
    for all but finitely many $n$.
\end{theorem}

Hence, even pathwise, the empirical distribution converges at a rate of $O(\log n/n)$, which is only a logarithmic factor worse than the optimal $O(1/n)$ rate,
and drastically improves over the rate of $O(1/\sqrt n)$ for IID sampling.

\subsection{Interpretation as an optimization problem}

\label{sec:optimization_problem}

Here, we formalize the interpretation of the self-balancing sampler as optimizing the trade-off between 
the convergence rate of the empirical distribution and its unpredictability at each step.
 To this end, we show that this sampling rule solves a greedy optimization problem, attempting 
 to maximize both these quantities at every step. Proofs can be found in Appendix~\ref{appendix:2_4_optimization_proofs}.

Let $H(q) \coloneqq -\sum_{i=1}^d q_i \log q_i$ denote the Shannon entropy, let $H(q,p) \coloneqq - \sum_{i=1}^d q_i \log p_i$ denote the cross-entropy, 
and let $\operatorname{KL}(q \| p) \coloneqq H(q,p) - H(q) = \sum_{i=1}^d q_i \log \frac{q_i}{p_i}$.

\begin{theorem} \label{thm:constant_optimization_problem}
    Let $L(n) \coloneqq \frac{1}{2} \sum_{i=1}^d (N_i(n) - \frac{n}{d})^2$ denote the quadratic loss for the empirical counts and $D(q) \coloneqq  \mathbb E[ L(n+1) - L(n) \mid N(n), q]$ denote the one-step increase in $L$ when sampling $X_{n+1} \sim q$. Then the unique solution to the optimization problem
    $$\argmax_{q \in \Delta^d} \left(H(q) - \beta D(q)\right)$$
    is the self-balancing sampling method $p_i(n) \propto e^{-\beta N_i(n)}$, initialized at the uniform distribution.
\end{theorem}

Theorem~\ref{thm:constant_optimization_problem} is a special case of the following more general version, which 
allows the $\alpha_i$ (and hence the $\beta_i$) to be non-constant in $i$. 
\begin{theorem} \label{thm:general_optimization_problem}
    Let $L_\beta(n) \coloneqq \frac{1}{2} \sum_{i=1}^d \beta_i (N_i(n) - np_i^*)^2$ denote the weighted quadratic loss for the empirical counts and $D_\beta(q) \coloneqq  \mathbb E[ L_\beta(n+1) - L_\beta(n) \mid N(n), q]$ denote the one-step increase in $L_\beta$ when sampling $X_{n+1} \sim q$. Then the unique solution to the optimization problem
    $$\argmax_{q \in \Delta^d} \left(H(q) - H(q,p(0)) - D_\beta(q)\right) = \argmin_{q \in \Delta^d} \left(\operatorname{KL}(q \| p(0)) + D_\beta(q)\right)$$
    is the $1/2$-offset self-balancing sampling method $p_i(n) \propto p_i(0)e^{-\beta_i (N_i(n)+1/2)}$, initialized at a tilted distribution $\propto p_i(0) e^{-\beta_i/2}$.
\end{theorem}

The entropy-regularized optimization problem found in Theorem \ref{thm:general_optimization_problem} appears naturally in the context of optimal control and reinforcement learning. In this setting, such regularization is used to quantify the classical exploration versus exploitation trade-off \cite{so2022maximumentropydifferentialdynamic}, \cite{WangZariphopoulouZhou2020}, \cite{haarnoja2018softactorcriticoffpolicymaximum}. Moreover, as \cite{NIPS2006_d806ca13} observes, such exponential reweighting naturally appears as a result of this entropy penalty. We briefly summarize some of the main results in discrete-time in Section \ref{appendix:optimal_control} below, and show that the self-balancing sampler and the above entropy-regularized optimization problem arise naturally in this framework when you penalize imbalance using mean-squared error.

The curious offset of $1/2$ in the counts appears to be an effect of greedy optimization over \emph{discrete} time steps with a quadratic cost. Indeed, $\frac{\beta_i}{2}((x+1)^2 - x^2) = \beta_i x + \frac{\beta_i}{2}$. When all $\beta_i$ are equal, this factor is absorbed into the normalization of the density. Theorem~\ref{thm:instantaneous_optimization} gives an ``instantaneous'' optimization perspective, which does not include the $1/2$ offset.

\begin{theorem} \label{thm:instantaneous_optimization}
    Let $L_\beta(x) \coloneqq \frac{1}{2} \sum_{i=1}^d \beta_i (x_i - n p_i^*)^2$ denote the weighted quadratic loss as a function of counts $x$. Then the unique solution to the optimization problem
    $$\argmin_{q \in \Delta^d} \left(\langle \nabla L_\beta(N(n)),q \rangle + \operatorname{KL}(q \|p(0)) \right)$$
    is the self-balancing sampling method $p_i(n) \propto p_i(0)e^{-\beta_i N_i(n)}$, initialized at the distribution $p_i(0)$.

    \label{thm:mirror_descent_optimization}
\end{theorem}

It is somewhat remarkable that the same sampling rule characterized by a global invariance rule (as in Proposition~\ref{prop:charactizing_v_invariance}) can also be uniquely characterized by a myopic, one-step optimization. However, Theorem~\ref{thm:instantaneous_optimization} can be viewed as an instance of the usual Gibbs variational principle \cite{https://doi.org/10.1002/cpa.3160280102}: a linear cost regularized by relative entropy has an exponentially weighted optimizer. The cost assigned to index $i$ is $\beta_i (N_i(n) - np_i^*)$, and since $\beta_i p_i^*$ is constant in $i$, the target-centering term does not affect the optimizer after normalization. The resulting policy therefore depends only on the weighted counts $\beta_i N_i(n)$.

Finally, we provide a third interpretation of the sampler: we observe that the self-balancing sampler implements stochastic mirror descent \cite{doi:10.1137/070704277}, minimizing the objective $\ell_\beta(x) \coloneqq \frac{1}{2} \sum_{i=1}^d \beta_i (x_i - p_i^*)^2$. The sampler uses the unbiased stochastic gradient $\beta_{X_n} e_{X_n} - \frac{1}{\sum_{j=1}^d \beta_j^{-1}} \mathbf 1$, where $e_1,\dots,e_d$ are the standard basis vectors:
$$\mathbb E\left[\beta_{X_n}e_{X_n} - \tfrac{1}{\sum_{j=1}^d \beta_j^{-1}} \mathbf 1 \mid p(n-1)\right] = \sum_{i=1}^d \beta_i(p_i(n-1) - p_i^*) e_i = \nabla \ell_\beta(p(n-1)).$$

\begin{theorem}[Stochastic mirror descent formulation]\label{thm:stochastic_mirror_descent}
    For fixed $h > 0$, the unique solution to the optimization problem
    $$\argmin_{q \in \Delta^d} \left(h \left\langle \beta_{X_n} e_{X_n} - \tfrac{1}{\sum_{j=1}^d \beta_j^{-1}} \mathbf 1, q \right \rangle + \operatorname{KL}(q \| p(n-1)) \right)$$
    is the self-balancing sampler $p_i(n) \propto p_i(0) e^{-h\beta_i N_i(n)}$.
\end{theorem}

Theorem~\ref{thm:stochastic_mirror_descent} shows that the elementary ``downweight the selected coordinate and renormalize'' operation is precisely an entropic mirror-descent step \cite{BECK2003167}. In the standard mirror-descent identity, $q_i(n) \propto p_i(n) e^{-h g_i}$, where $g$ is a stochastic gradient. In our case, the stochastic gradient is supported only on the sampled coordinate $X_n$. Hence the mirror-descent update simply downweights the probability of the outcome which was just sampled and then renormalizes.

From the usual optimization perspective, the constant step-size $h$ may seem strange, since it generally means that the stochastic mirror descent algorithm may not converge. In our case, this is expected, since the sequence $p(n)$ does not converge. Considering a step size decreasing in $n$, as in Theorem~\ref{thm:time-dep-conv-rate}, can lead to a version of the sampler in which $p(n)$ (and the associated stochastic mirror descent) does converge.

The role of entropy here is also slightly different from its usual interpretation in reinforcement learning. There, entropy is often used to encourage exploration while learning an unknown environment \cite{haarnoja2018softactorcriticoffpolicymaximum}. Here, the environment is known, and the entropy term instead prevents the sampling rule from becoming overly predictable. Thus the corresponding trade-off is between reducing the current imbalance and preserving randomness in the next sample.


\subsection{Expected cover time}

Beyond asymptotic convergence of the empirical distribution, it is also important in applications 
that the sampler does not neglect any index for too long. A natural way to quantify this is via the cover 
time, namely the number of steps required before every outcome has been sampled at least once.
 In audit and inspection settings, this measures how long a team can go entirely untested; 
 in sampling and representation problems, it measures how long a category can remain completely unrepresented in the realized sample. Proofs can be found in Appendix~\ref{appendix:2_5_expected_cover_time}.

For deterministic sampling, the number of steps to see all $d$ outcomes is $d$, and for IID sampling, the expected number of steps to see all the outcomes, known as the ``coupon collector problem,'' is $\Theta(d\log d)$. The self-balancing sampler has behavior which interpolates between these two extremes, as we now show.

\begin{theorem}\label{thm:expected_cover_time}
    Let $C \coloneqq \max_{1 \leq i \leq d} \min \{ n \geq 1 : X_n = i \}$ be the number of steps until the self-balancing sampler sees every outcome at least once. Then, starting from the initial condition $p_i(0) = 1/d$,
    $$\E[C] = \sum_{i=1}^d \frac{\log(1 + \beta_i \log d)}{\beta_i} + O(d).$$
\end{theorem}

For fixed $\beta_i$, the cover time is $\Theta(d \log \log d)$. Moreover, in the extremal cases $\beta \uparrow \infty$ and $\beta \downarrow 0$, the self-balancing sampler's behavior recovers the deterministic and IID cover-time rates, respectively:
$$\E[C] =\Theta(d) \text{ as $\beta \uparrow \infty$}, \qquad \E[C] = \Theta(d \log d) \text{ as $\beta \downarrow 0$}.$$


\section{Diffusion Approximation}
\label{sec:diffusion_approximations}

In this section, we explore a few different diffusion approximations of our self-balancing sampler, highlighting its mean-reversion and other asymptotic properties. We are primarily interested in the \textit{weak-biasing regime} where $\beta \downarrow 0$. Since for fixed $\beta$ the centered counts process is essentially bounded, this is the natural regime under which a diffusive limit can arise.

 Our first result shows that in this weak-biasing regime each coordinate is in a sense just a time-changed Ornstein–Uhlenbeck (OU) process. Proofs can be found in Appendix~\ref{appendix:3_1_diffusion_limit_proofs}.

\begin{proposition}\label{prop:projection_centered_diffusion_limit}
    
Fix $i \in [d-1]$ and let $\tilde{Y}_i^\beta(n) := \beta_iN_i(n) - \beta_dN_d(n)$ be the projection-centered version of the self-balancing sampler with the given $\beta$. Let:
\[\tau^\beta_k := \inf \Big\{n > \tau^\beta_{k-1} \; : \; \tilde{Y}_i^\beta(n) \neq \tilde{Y}_i^\beta(\tau^\beta_{k-1}) \Big\}, \quad \tau_0^\beta = 0.\]
For $h > 0$, define the continuous-time chain:
\[Z^h_t := h^{-1/2} \left(\tilde{Y}_i^{h\beta}(\tau^{\beta h}_{\lfloor t/h\rfloor}) - c_\beta\right), \quad c_\beta := \log\left(\frac{p_i(0)}{p_d(0)}\right) - \log\left(\frac{p^*_i}{p^*_d}\right).\]
\noindent If $Z_0^h \to z_0$, then we have $Z^h_t \Rightarrow Z_t$ as $h \downarrow 0$ where $Z_t$ is the diffusion
\[dZ_t = -\frac{\beta_i \beta_d}{\beta_i + \beta_d} Z_t dt + \sqrt{\beta_i\beta_d} \cdot dW_t, \quad Z_0 = z_0.\]
\noindent Here $W_t$ denotes a standard 1-dimensional Brownian motion and $\Rightarrow$ denotes convergence in the Skorokhod space $D([0,T], \R)$ for any $T < \infty$. 
\end{proposition}

In the uniform case $\beta_i \equiv \beta$, the diffusion simplifies to
\[dZ_t = -\frac{\beta}{2} Z_t dt + \beta dW_t,\]
\noindent a simple OU process. The OU process is well-known for its strong mean-reversion properties, so in a sense it is unsurprising it appears as a limit of our discrete self-balancing sampler.

The correction $c_\beta$ shifts the starting point of $\tilde{Y}(n)$ to the location
where it has zero drift and where the sampling law $p(n)$ agrees with the target law $p^*$.
Except in the case where $p(0) = p^*$ and thus $c_\beta = 0$, this zero-drift location is \textit{not} the natural
initial condition where all counts are zero. Consequently, Proposition~\ref{prop:projection_centered_diffusion_limit}
should not be interpreted as a diffusion limit from the all-zero count
initialization. Rather, it describes the local fluctuations of the sampler near
the zero-drift point, or heuristically the behavior after the process has already
entered its equilibrium scale. In a sense, $c_\beta$ represents a penalty we pay for starting our process
away from stationarity. As we saw in Theorem~\ref{thm:average_case_conv_rate}, this penalty has a strong
impact on leading-order constants of the convergence rate.

We develop an analogous diffusion-approximation for the mean-centered counts, this time without needing a time-change. Here, the interaction between the different coordinates becomes visible, and again the process exhibits mean-reversion.


\begin{proposition}
Let $M^\beta(n) := (\beta_i[N_i(n) - np_i^*])_{i=1}^d$ be the (weighted) mean-centered version of the self-balancing sampler with the given $\beta$. For $h > 0$, define the continuous-time chain:
\[Z^h_t := h^{-1/2}\left(M^{h\beta}(\lfloor t/h\rfloor) - c_\beta\right), \quad (c_\beta)_i := \log\left(\frac{p_i(0)}{p_i^*}\right) + \mathrm{KL}(p^* \mid p(0)).\]
\noindent If $Z_0^h \to z_0$, then $Z^h_t \Rightarrow Z_t$ as $h \downarrow 0$ where $Z_t$ is the diffusion:
\[dZ_t = -\frac{1}{S_\beta}Z_t dt + \Sigma^{1/2} d\mathbf{W}_t, \quad Z_0 = z_0,\]
\noindent for $S_\beta = \sum_j \beta_j^{-1}$  and $\Sigma = S_\beta^{-1}\mathrm{diag}(\beta)  - S_\beta^{-2} \mathbf{1}\mathbf{1}^T$. Here $\mathbf{W}_t$ denotes a d-dimensional standard Brownian motion.

\label{prop:mean_centered_diffusion_limit}
\end{proposition}

\noindent As a consequence of the centering, this diffusion is constrained to the hyperplane
$$\{z : \langle z, p^* \rangle = 0\}.$$

\noindent Again, the situation simplifies considerably in the uniform $\beta_i \equiv \beta$ case. 
\[dZ_t = -\frac{\beta}{d} Z_t dt + \frac{\beta}{d}(d I_d  - \mathbf{1} \mathbf{1}^T)^{1/2} dW_t.\]
\noindent The centering $c_\beta$ in Proposition~\ref{prop:mean_centered_diffusion_limit} satisfies
\begin{equation}
    \langle c_\beta, p(0) \rangle = \mathrm{KL}(p(0) \mid p^*) + \mathrm{KL}(p^* \mid p(0)),
    \label{eqn:c_beta_and_sym_KL_divergence}
\end{equation}

\noindent the symmetrized KL divergence. Hence again $c_\beta$ is a measure of the typical discrepancy between the initial sampling law $p(0)$ and the stationary sampling law $p^*$. 

\subsection{Interpreting Diffusive Limits}

Using our diffusion limits, we can study how our two centered-counts processes behave in this weak-biasing regime. This section is non-rigorous and mainly serves to provide some evidence the upper bounds appearing in Theorem~\ref{thm:average_case_conv_rate} are in a sense tight. Again, we emphasize the diffusion limits model the fluctuations of the process about the zero drift location, which is not the natural starting condition where all counts are zero except when $c_\beta = 0$. Moreover, we will study these limiting diffusions at their stationarity, but the discrete centered chains may only have a stationary distribution in the rational-ratio setting of Lemma~\ref{lem:projection_centered_has_stationarity}.

We start with the projection-centered case. Moreover, by studying the process's autocorrelation we aim to provide an alternative measure of its predictability in contrast to the adversarial notion discussed in Section~\ref{sec:adversarial_interpretation}. 

Below, we will rely on some basic facts about OU processes, details of which can be found in Section 6.5 of \cite{risken1996fokker}. Recall the one-dimensional OU process
\[dZ_t = -\theta Z_t dt + \sigma dW_t, \quad Z_0 = z_0,\]
\noindent has stationary distribution $N(0, \sigma^2/2\theta)$ and autocorrelation
\[\mathrm{Corr}(Z_t, Z_{t+s}) = e^{-\theta s}.\]
\noindent For small $\beta$, where sampling is close to uniform, the waiting times $\tau^\beta_k - \tau_{k-1}^\beta \approx \mathrm{Geo}(2/d)$, so we wait roughly time $d/2$ between samples. Hence using Proposition \ref{prop:projection_centered_diffusion_limit}, we heuristically see that in the uniform $\beta$ case the (time-changed) projection-centered counts are well-approximated by: 
    \[N_i(n) - N_d(n) \approx \frac{Z_{2nh/d}}{\beta\sqrt{h}}.\]
\noindent which is an OU process with stationary distribution $N(0, (\beta h)^{-1})$. At stationarity:
\[\mathrm{Corr}(N_i(n+k) -  N_d(n+k), N_i(n) -  N_d(n)) \approx \exp\left(-\frac{\beta hk}{d}\right). \]
\noindent Since $\beta h$ was just our rescaled bias, this essentially says in the small $\beta$ regime at stationarity the process's auto-correlations decay exponentially at rate $O(\beta)$. Since $p^*(\beta) = p^*(c\beta)$ for any $c > 0$, scaling $\beta$ has no effect on the limiting distribution. 

For the mean-centered counts, the story is similar. Specifically, for arbitrary $d\times d$ matrices $\Theta, \Sigma$ the $d-$dimensional OU process
\[dZ_t = -\Theta Z_t dt + \Sigma dW_t,\]
\noindent has stationary distribution $N(0, \omega)$ for $\omega$ solving the Lyapunov equation
\[\Theta \omega + \omega \Theta^T = \Sigma \Sigma^T.\]
\noindent Moreover, stationary covariance is
\[\mathrm{Cov}(Z_{t+s}, Z_t) = e^{-\Theta s} \omega\]
\noindent In the context of Proposition \ref{prop:mean_centered_diffusion_limit}, it is easy to see $\Theta = S_\beta^{-1} I_d$ and
\[\omega = \frac{1}{2} [D_\beta - S_\beta^{-1} 11^T] = \frac{1}{2}D_\beta [I_d - p^*1^T],\]
\noindent for $D_\beta = \mathrm{diag}(\beta_i)$. Thus, Proposition \ref{prop:mean_centered_diffusion_limit} implies
\[N(n) - np^* \approx D_\beta^{-1} \left[\frac{c_\beta}{h} + \frac{Z_{nh}}{\sqrt{h}} \right] \sim N\left( \frac{D_\beta^{-1}c_\beta}{h},\;  \frac{D_\beta^{-1}\omega D_\beta^{-1}}{h}\right).\]
\noindent At stationarity, for any nonzero $a \in \R^d$ we have autocorrelation
\[\mathrm{Corr}(a^T(N(n+k) - (n+k)p^*), \; a^T(N(n) - np^*)) = e^{-k h / S_\beta},\]
Hence the auto-correlation decay rate is $O(\beta)$.

Note
\[\E ||\mu_n - p^*||_{TV} = \frac{1}{2n} \sum_i |N_i(n) - np_i^*| = \frac{\E ||Y(n)||}{2n}.  \]

\noindent Using the above stationary distribution, our diffusion approximation suggests
\[\E ||\mu_n - p^*||_{TV} \approx \frac{1}{2n}\left(\sum_{i=1}^d \E \left|N\left(\frac{(c_\beta)_i}{\beta_i h}, \;  \frac{1-p_i^*}{2\beta_ih}\right)\right| \right).\]
\noindent Hence, the mean of a folded-normal
\[\E |N(\mu, \sigma^2)| =
\sigma \sqrt{\frac{2}{\pi}}
\exp\!\left(-\frac{\mu^2}{2\sigma^2}\right)
+
\mu\left[1-
2\Phi\!\left(-\frac{\mu}{\sigma}\right)
\right],
\]
\noindent suggests there are two regimes:
\[\E ||\mu_n - p^*||_{TV} \approx \begin{cases}
    O\left(\frac{d}{n\sqrt{\beta h}}\right) & c_\beta = 0,\\[5pt]
    O\left(\frac{1}{nh} \sum \frac{|(c_\beta)_i|}{\beta_i}\right)& c_\beta \neq 0.\\
\end{cases}\]
\noindent Since $c_\beta = 0$ if and only if $p(0) = p^*$, we recover the same leading order constants appearing in the upper bound Theorem~\ref{thm:average_case_conv_rate}. This suggests these constants are indeed tight, and we supplement this with further empirical evidence in Appendix \ref{appendix:empirical_results}.

\subsection{Adversarial Interpretation of Predictability}
\label{sec:adversarial_interpretation}

In order to develop a notion of how
``predictable'' such a sampler is, consider the case where an adversary attempts to guess the next sample given the current sample history. Then one way of quantifying the predictability of the sampler is to study how accurate such an adversary can be if they behave optimally. In this section, we formalize this idea. Proofs can be found in Appendix~\ref{appendix:3_3_adversarial_error_rate_proof}.

Predictability criteria of this form have been studied extensively for restricted randomization designs, where selection bias is commonly measured by the probability that an observer can correctly guess the next treatment assignment \cite{05287b01-9a36-3190-b702-371959e17200,7708aeae-3258-37ff-b421-c7c45f9f9dd0,https://doi.org/10.1111/1467-985X.00564}. When $d=2$, $\beta$ is uniform, and $p(0) = p^*$ we show that the prediction accuracy is order $1/2  + \Theta(\sqrt{\beta})$ in the weak-biasing $\beta \downarrow 0$ regime. We conjecture this is true more generally, and using our diffusive limits, we provide heuristic evidence this accuracy is order $\max_i p_i^* + O(\sqrt{\beta})$ when $p^*$ does not have a unique maximizer in the same weak-biasing regime. 


Clearly, since the sampler exponentially favors under-sampled indices, the best strategy for the adversary is to just guess
the most under-sampled index. Under this strategy, we define the adversary's asymptotic error-rate as
\[E_\beta := \limsup_{n \to \infty} \frac{1}{n}\sum_{k=0}^{n-1}\P\!\left(X_{k+1} \neq \argmax_{i\in [d]} p_i(N(k))\right) = 1-\liminf_{n \to \infty} \frac{1}{n}\sum_{k=0}^{n-1} \E_{y \sim \tilde{Y}(k)} \left[\max_{i\in[d]} p_i(y)\right]\]

\noindent where $\tilde{Y}(k)$ denote our projection-centered counts. Lemma \ref{lem:projection_centered_has_stationarity} shows, at least in the rational-ratio setting, these centered counts form an ergodic Markov chain, so it has a stationary distribution $\eta$. Moreover, the ergodic theorem implies the finite-sample error-rate ought to converge to the expected error-rate at stationarity:
    \[E_\beta = 1 - \E_{y \sim \eta} \left[\max_{i\in[d]} p_i(y)\right].\]
In general, it is hard to determine what the stationary distribution of the projection-centered counts $\tilde{Y}$ is. However, in the special case of $d=2$, then $\tilde{Y}$ becomes a simple birth-death process on $\Z$ which makes this a lot more tractable. 

\begin{lemma}
    When $d=2$ and uniform $\beta_1 = \beta_2 = \beta > 0$, the projection-centered counts $\tilde{Y}(n)$ form a reversible Markov chain on $\beta\Z$ with stationary distribution:
    \[\eta(\beta z) \propto  \frac{p_2(0) + p_1(0)e^{-\beta z}}{p_1(0) + p_2(0)} \left(\frac{p_1(0)}{p_2(0)}\right)^z e^{-\beta {z \choose 2}}\]
    \label{lem:2d_stationary_project_centered}
\end{lemma}

\noindent In the case $p_1(0) = p_2(0) = \frac{1}{2}$, we get the symmetric and especially simple form:
\[\eta(\beta z) \propto \cosh\left(\frac{\beta z}{2}\right)  e^{-\frac{\beta z^2}{2}}.\]

\noindent In the $d=2$ case, we can quantify how this error-rate depends on $\beta$ in the weak-biasing regime.

\begin{lemma}\label{lem:small_beta_adversarial}
    When $d=2$ and uniform $\beta$, as $\beta \downarrow 0$ we have:
    \[E_\beta = \frac{1}{2} - \Theta(\sqrt{\beta}). \]
\end{lemma}

\noindent More precise constants can be found in Appendix \ref{appendix:proofs}. Even when $d=2$, if $\beta$ is non-uniform we lose this simple birth-death structure. Instead, to analyze the general $\beta$ case we rely on our diffusive limits. Again, the following section is non-rigorous, but aims to add some heuristic evidence for why this $\Theta(\sqrt{\beta})$ rate appears. 

We saw in Section \ref{sec:diffusion_approximations} that in the weak-biasing regime the weighted mean-centered counts are approximated at stationarity by
\[M^\beta(n) = D_\beta [N(n) - np^*] \approx c_\beta + \sqrt{h} Z, \qquad Z \sim N(0,\omega),\]
\noindent where $\omega$ is the stationary covariance from Section \ref{sec:diffusion_approximations}. Since
\[p_i(M) = \frac{p_i(0)e^{-M_i}}{\sum_{j=1}^d p_j(0)e^{-M_j}}, \qquad p(c_\beta) = p^*,\]
\noindent this suggests
\[\E\left[\max_i p_i(n)\right] \approx \E\left[\max_i p_i(c_\beta + \sqrt{h} Z)\right].\]
\noindent Let $J_\beta$ denote the Jacobian of the softmax map at $c_\beta$. A direct calculation gives
\[J_\beta = -\left[\operatorname{diag}(p^*) - p^*(p^*)^T\right],\]
\noindent and hence the first-order Taylor expansion yields
\[p(c_\beta + \sqrt{h} Z) = p^* + \sqrt{h}J_\beta Z + O(h).\]
\noindent so heuristically
\[\E\left[\max_i p_i(n)\right] \approx \E\left[\max_i \left(p_i^* + \sqrt{h} G_i\right)\right], \quad G = J_\beta Z \sim N(0, J_\beta \omega J_\beta^T).\]

\noindent If the maximum of $p^*$ is not unique, then these first-order Gaussian fluctuation contribute. Writing $I^* = \{i : p_i^* = \max_j p_j^*\}$ for the set of maximizers, a first-order expansion gives
\[\E\left[\max_i p_i(n)\right] = \max_i p_i^* + \sqrt{h}\,\E\left[\max_{i \in I^*} G_i\right] + O(h).\]
\noindent Using the fact $\langle Z, p^* \rangle  = 0$, we see:
\[G = -\mathrm{diag}(p^*) Z \overset{d}{=} -\mathrm{diag}(p^*)(W - \langle W, p^* \rangle 1) \]

\noindent for $W \sim N(0, D_\beta / 2)$. Hence as the second term is constant on $I^*$ and mean-zero
\[\E\left[\max_i p_i(n)\right] - \max_i p_i^* = \sqrt{\frac{h}{2\beta_{min} S_\beta^2}} \cdot \E \max_{i \leq |I^*|} B_i + O(h)  \leq \sqrt{\frac{h \log | I^*|}{\beta_{min} S_\beta^2}} + O(h).\]
\noindent for $B_i \overset{iid}{\sim} N(0,1)$. The second approximation comes from classical asymptotics for the maximum of iid Gaussians, and is approximately tight when $|I^*|$ is large. Hence 
\[E_\beta \approx 1-\max_i p_i^* - \sqrt{\frac{h \log | I^*|}{\beta_{min} S_\beta^2}} + O(h).\]
\noindent In the uniform $\beta$ case, this becomes
\[E_\beta \approx 1 - \frac{1}{d} - \frac{\sqrt{\beta h  \log d}}{d} + O(h)\]
 Again, this is reasonable when $d$ is large. Thus, the diffusion limit suggests in the weak-biasing regime, the predictability
of the sampler is governed primarily by the maximum of $p^*$, and by the number of coordinates sharing this maximum value. Increasing either increases the predictability of the sampler. Moreover, for fixed $p^*$, it shows this
 predictability grows at rate $\Theta(\beta^{1/2})$ as $\beta \downarrow 0$, which matches the $d=2$ case from Lemma \ref{lem:small_beta_adversarial}. Somewhat interestingly, predictability can increase when several coordinates share the minimal value of \(\beta_i\), since the adversary can exploit the largest realized fluctuation among these tied coordinates. 
 

\subsection{Optimal Control}
\label{appendix:optimal_control}

In this section, we briefly review some basic optimal control theory, and discuss how Theorem \ref{thm:general_optimization_problem} fits into the entropy-regularized optimal control framework. Suppose we have some discrete time dynamical system on some state space $S$:
\[x_{t+1} = b(x_t, u_t), \quad x_t \in S, u_t \in U\]
\noindent where $u_t$ is the input of some controller and $U$ is the set of all allowable controls. In classical control
theory, we have some loss function $\ell(x,u)$ and we wish to minimize the cumulative (discounted) loss over the entire
trajectory $(x_t)_{t \in \N}$. Namely we aim to solve
\[V(x) := \inf_{(u_t)} \left[\sum_{s=1}^\infty \gamma^s \cdot \ell(x^u_s, u_s) \right],\]
\noindent for some discount $\gamma \in (0,1)$ and where $x_0 = x$ is the initial state of the system. Here $(x^u_t)$ is the trajectory of the process when control sequence
$(u_t)$ is applied. The \textit{value function} $V$ and optimal control $(u^*_t)$ can be determined
using the standard Bellman equation:
\[V(x) = \inf_{u \in U} \left\{\ell(x,u) + \gamma V( b(x,u))\right\}.\]
\noindent If we let $Q_\gamma(x,u) := \ell(x, u) + \gamma V(b(x,u))$ then we see the optimal policy $u^*$ and its
corresponding value function $V$ satisfy
\begin{equation} 
	u^*(x) = \arg\min_{u \in U} Q_\gamma(x, u), \quad V(x) = Q_\gamma(x, u^*(x)).
	\label{eqn:optimal_unregularized_policy}
\end{equation}

We can generalize the above by allowing the control policy to select its control $u_t$ at random, possibly depending on the current state. This paradigm is useful in the following settings:

\begin{enumerate}[label=(\roman*)]

    \item \textit{Adversarial or Strategic Environments}: In some settings, predictability can be exploited by an adversary or the environment. Entropy regularization encourages policies that remain sufficiently randomized, thereby mitigating worst-case or strategic responses to deterministic behavior \cite{husain2021regularized}.

    \item \textit{Exploration Versus Exploitation}: In reinforcement learning, we often want to encourage the agent to explore the state space in order to learn about the environment and the corresponding reward process. Randomized policies explicitly encode this exploration into the objective function.

    \item \textit{Robustness and Regularization}: Entropy regularization acts as a convex regularizer on the control problem, smoothing the optimization landscape and preventing degenerate (overly deterministic) solutions. This can significantly improve stability and sample efficiency in practice \cite{haarnoja2018softactorcriticoffpolicymaximum}.

    \item \textit{Information Constraints / Bounded Rationality}: Entropy regularization can also be interpreted as penalizing deviations from a reference policy via a Kullback--Leibler divergence. In this view, the controller faces a cost for using highly concentrated policies, reflecting limited information-processing capacity or a preference for staying close to a prior behavior (\cite{ortega2015informationtheoreticboundedrationality}). 
\end{enumerate}

 Let $U$ denote the set of available controls. Suppose there is some deterministic function $\pi(\cdot \mid x) : S \to \mathcal{P}(U)$ which maps the state-space $S$ into the set  $\cP(U)$ of probability distributions on the set of controls. Define dynamics
\[X_{t+1} \mid X_t  \sim b(X_t, U_t), \quad \P(U_t \in \cdot \mid \{X_i\}_{i \leq t}, \{U_i\}_{i < t} ) =  \pi(\cdot \mid X_t), \quad X_0 = x_0.\]
\noindent Namely, at time $t$ our state $X_t$ determines a distribution $\pi(\cdot \mid X_t)$ over the set of controls, and conditional on $X_t$ we sample a new control independently from this distribution. This generalizes the Markovian sampler we discuss above. We then force a trade-off between the entropy of the control and the accrued loss:
\[V(x_0) := \inf_{(\pi)} \E\left[\sum_{s=0}^\infty \gamma^s \cdot (\ell(X^u_s, U_s) - \tau H(\pi(\cdot \mid X_s)))\right],\]
\noindent In this case, Bellman's equation gives:
\[V(x) = \inf_{\mu \in \mathcal{P}(U)} \left\{\int_U \left[Q_\gamma(x,u) + \tau \log(\mu(u))\right] \; \mu(du)  \right\}.\]
\noindent This can be optimized to find the optimal policy $\pi$ pointwise, as well as the value function. In the case $U$ is finite,
these solutions can be expressed implicitly as
\begin{equation}
	\pi^*( u \mid x) \propto \exp\left(- \frac{Q_\gamma(x,u)}{\tau}\right), \quad V(x) = -\tau \log \left(\sum_{u \in U} \exp\left[- \frac{Q_\gamma(x,u)}{\tau}\right]\right).
	\label{eqn:optimal_entropy_regularized_policy}
\end{equation}

\noindent As we can see, the optimal policy is naturally a soft-min / Gibbs distribution with Hamiltonian $-Q_\gamma(x, \cdot ) / \tau$, and the value function
is essentially the log-partition function of this distribution. Note how entropy-regularization essentially smooths out
the strict minimization problem that occurs in equation \eqref{eqn:optimal_entropy_regularized_policy}. While equations \eqref{eqn:optimal_entropy_regularized_policy} only implicitly define the optimal policy, for certain choices of dynamics and loss functions we can be more explicit. Suppose we have some potential $\Phi:S \to \R$ and we define loss:
\[\ell(x, u)  := \Phi(x) +  \tau \log \sum_{v \in U} \exp\left(-\frac{\gamma}{\tau} \cdot \Phi(b(x,v))\right),\]
\noindent which is independent of $u$. Note that the second term is a simple soft-minimum. That is, we penalize spending time in states $x$ that have high-potential $\Phi(x)$ and that do not have neighboring low-potential states. In this setting:
\[\pi^*( u \mid x) \propto \exp\left(- \frac{\gamma(\Phi(b(x,u)) - \Phi(x))}{\tau}\right), \qquad V(x) = \Phi(x).\]
The self-balancing sampler fits nicely into this framework. If we choose state-space $S = \N^d$, controls $u \in [d]$, dynamics $b(x,u) = x + \delta_u $, and potential $\Phi(x) = \frac{1}{2}||x||^2$ then we have $\pi^*(u | x) \propto e^{-\gamma x_u/\tau}$. Hence we recover optimization problem Theorem \ref{thm:constant_optimization_problem}, the uniform $\beta_i \equiv \beta$ case. Analogously, if we instead choose $\Phi(x) = \frac{1}{2} \sum_i \beta_i(x_i - np_i^*)^2$ we recover optimization problem \ref{thm:general_optimization_problem}.

The former gives us another way of interpreting the role of $\beta$ in the self-balancing sampler: $\beta = \gamma / \tau$ measures the trade-off between minimizing future imbalance through the value term $Q_\gamma$ and maintaining randomness in the sampling policy via entropy-regularization.

Alternatively, if we choose a linear loss $\ell(x,u) = \beta_u x_u$, and myopic discounting $\gamma=0$, we see the policy $\pi^*(u \mid x) \propto \exp(-\beta_ux_u)$ is the minimizer of:
\[V(x) := \inf_\pi \left\{\sum_{i=1}^d [\pi(i|x)  \beta_ix_i] - H(\pi(\cdot \mid x)) \right\},\]
\noindent Hence the self-balancing sampler can be viewed as an entropy-regularized greedy mechanism that favors actions having smaller weighted current counts $\beta_i x_i$, analogous to Theorem \ref{thm:mirror_descent_optimization}.

We showed above that the self-balancing sampler converges at rate $O(n^{-1})$ on average. The theory of entropy-regularized optimal control provides us a convenient framework for extending these ideas to a broader class of samplers. The following shows as long as under-sampled indices are at least exponentially favored, we get the same convergence rate.

\begin{proposition}

Suppose we have sampler $\pi_i(x) \propto a_i\exp(-H_i(x))$ so that for $P = I - p^*1^T$ then $\pi(x) = \pi(x')$ whenever $Px=Px'$. Moreover, suppose there are constants $c,C > 0$ so for all $y$ with $\sum y_i = 0$ and $y_j \leq y_i$ we have: 
\[H_i(y) - H_j(y) \geq c(y_i-y_j) - C.\]
\noindent Then if $\mu_n$ is the empirical distribution associated to sampler $\pi$, we have:
\[\E || \mu_n - p^*||_{TV} \leq O(n^{-1}).\]
\label{prop:average_case_optimal_control}
\end{proposition}

\noindent As in the previous sections, $p^*$ plays the role of the limiting distribution. Our assumptions on $\pi$ simply state that it is a function of the mean-centered counts alone, and the proper centering $p^*$ determines the limiting distribution. 

If we have a sampler generated by potential $\Phi$ as above, then the assumptions of Proposition \ref{prop:average_case_optimal_control} reduce to the constraints:
\[\Phi(x + \delta_i) - \Phi(x) = \Phi(Px + \delta_i) - \Phi(Px) + \kappa,\]
\[\Phi(x + \delta_i) - \Phi(x+\delta_j) \geq c((Px)_i - (Px)_j) - C,\]
for some $\kappa$ independent of $x, i$. The self-balancing sampler is the special case $H_i(y) = \beta_i y_i$ which is globally linear and arises from potential $\Phi(x) = \frac{1}{2}||x||^2$.

\section{Concluding Remarks}

We introduced a family of self-balancing sequential samplers that adaptively downweight previously sampled outcomes. The resulting process interpolates between IID sampling, which is maximally random but poorly balanced, and deterministic sampling, which is optimally balanced but perfectly predictable. Despite retaining randomness at every step, the self-balancing sampler achieves empirical convergence at the deterministic scale \(O(n^{-1})\) in expectation and \(O(\log n/n)\) almost surely.

A central theme of the paper is that this sampler is not merely a heuristic. It arises naturally from an invariance principle for potential-based Markovian samplers, is the unique solution to a one-step entropy-regularized imbalance minimization problem, implements a stochastic mirror descent on the space of next-step sampling distributions, and shows up informally in actual policy decisions. Together, these perspectives give complementary evidence that such exponential sampling rules arise naturally.

The weak-biasing regime makes the convergence--predictability tradeoff explicit. As the bias parameter decreases, the centered counts fluctuate on a larger scale and mix more slowly, preserving more unpredictability. As the bias parameter increases, the process becomes more tightly mean-reverting, improving balance but making the next sample easier to predict. The diffusion approximation quantifies this tradeoff and shows that the geometry of the target law matters: tied maximizers of \(p^*\) produce first-order Gaussian corrections to predictability, while unique maximizers with fixed gaps only contribute through rare crossing events.

Several extensions remain open. One natural direction is to replace categorical balance with covariate balance, sampling units according to their contribution to a cumulative covariate imbalance. Namely, suppose each unit $i \in [d]$ has some covariate vector $v_i$ assigned to it. If we have samples $X_1,..., X_n \in [d]$ define the cumulative (centered) covariate vector as:
    \[S_n := \sum_{i=1}^n v_{X_i} - n \bar{v}, \qquad \bar{v} := \frac{1}{d} \sum_{i=1}^d v_i.\]
\noindent We can then generalize our self-balancing sampler so that index $i$ is sampled with probability:
    \[p_i(n) \propto \exp\left(-\sum_{j} \beta_j v_{ij} (S_n)_j\right).\]
\noindent When $v_i = e_i$ are the standard basis vectors, this recovers our original self-balancing sampler. This provides a sequential analog to the Gram-Schmidt walk of \cite{BansalDadushGargLovett2019}, which in the context of \cite{Harshaw01102024} provides the key tool for balancing covariates across treatment groups in randomized controlled trials. 

Another is to study risk-weighted or nonstationary versions of the sampler, where the target distribution changes over time or depends on observed outcomes. Such extensions would bring the model closer to applications in inspections, audits, randomized experiments, and other settings where one seeks to balance coverage, responsiveness, and resistance to prediction. For instance, in both our health-inspections and anti-doping examples, follow-up inspections are scheduled more frequently for those performing poorly on the initial inspection.

Overall, we have shown that self-balancing exponential sampling rules provide a simple yet principled framework for designing sequential randomized procedures that balance coverage, unpredictability, and long-run empirical stability.

\paragraph*{Acknowledgements.} The authors gratefully acknowledge Sinho Chewi, Steven Evans, and Robin Pemantle for many helpful comments. Daniel Raban acknowledges financial support from Sandia National Laboratories, a U.S.\ Department of Energy multimission laboratory.

\newpage 

\setlength{\nomlabelwidth}{3cm} 
\renewcommand{\nomlabel}[1]{\hfil #1 \hfil} 

\setlength{\nomitemsep}{8pt} 

\printnomenclature

\newpage

\bibliographystyle{plainnat}
\bibliography{references}


\newpage 
\appendix
\section{Empirical Results}
\label{appendix:empirical_results}

In this section, we provide some additional empirical evidence for some of the $\beta$ scalings we developed heuristically in Section \ref{sec:diffusion_approximations} above. Simulation code can be found in the GitHub repository \href{https://github.com/zackmcnulty/self_balancing_sampler}{self\_balancing\_sampler}. 

In the first few sections, we study the behavior of the self-balancing sampler under two different unit-norm choices of $\beta$ when $d=4$: the uniform $\beta$ and the non-uniform $\beta \propto [1.1, 0.8, 0.9, 1.2]$, the latter of which also has a unique maximizer of $p^*(\beta)$. In the last section, we give an example of how the self-balancing sampler can be applied to the citizen's assembly problem, and analyze its behavior in this setting.

\subsection{Convergence Rate}

 For $d_n := \E ||\mu_n - p^*||_{TV}$, Theorem \ref{thm:average_case_conv_rate} and the preceding remarks imply

\[\lim\sup_{m \geq n} md_m  = C(\beta).\]

\noindent The goal of this section is to determine how the bias parameter $\beta$ influences the coefficient of the leading $O(n^{-1})$ term. Theorem~\ref{thm:average_case_conv_rate} and our heuristic analysis in Section \ref{sec:diffusion_approximations} suggests
\[C(c\beta) = \begin{cases}
    \Theta(c^{-1/2}) & p(0) = p^*,\\
    \Theta(c^{-1}) & p(0) \neq p^*.
\end{cases}\]
We provide empirical evidence these are the correct scalings. To do so, we generate estimates $\hat{d_n}$ by simulating $M$ trajectories of the self-balancing sampler and taking the empirical average
    \[\hat{d}_N = \frac{1}{M} \sum_{i=1}^M\  ||\mu_N^{(i)} - p^*||_{TV},\]
\noindent for some large $N$. From here, we estimate
    \[\hat{C}(\beta) = N\hat{d}_N.\]
To get a sense of how this constant varies with $\beta$, we fix some unit-norm $\beta$ and study how $C(c\beta)$ varies as we vary the scaling $c > 0$. Since $p^*(\beta) = p^*(c\beta)$ for any $c>0$, all these samplers have the same limiting distribution.

For our simulation, we chose $M=1000$ and $N=10000$. The latter was chosen large enough so $N||\mu_N - p^*||_{TV}$ approximately stabilizes across the range of $c$ values tested. The results are displayed in Figure \ref{fig:convergence_plots}. Empirically, this supports the $\Theta(c^{-1})$ and $\Theta(c^{-1/2})$ scalings our analysis suggests, at least in the small $\beta$ regime.  

\begin{figure}[t]
    \centering

    \begin{subfigure}[t]{0.48\textwidth}
        \centering
        \includegraphics[width=\textwidth]{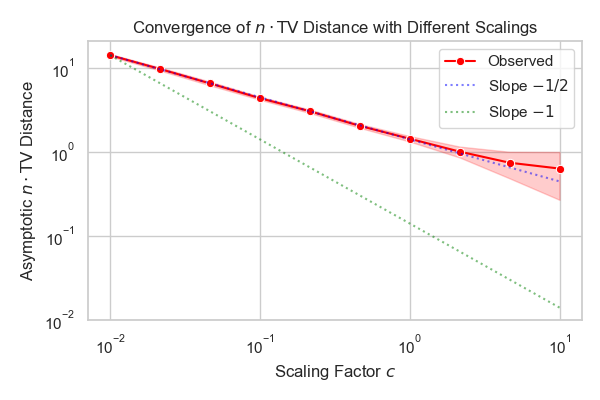}
        \caption{$p(0) = p^*$}
    \end{subfigure}
    \hfill
    \begin{subfigure}[t]{0.48\textwidth}
        \centering
        \includegraphics[width=\textwidth]{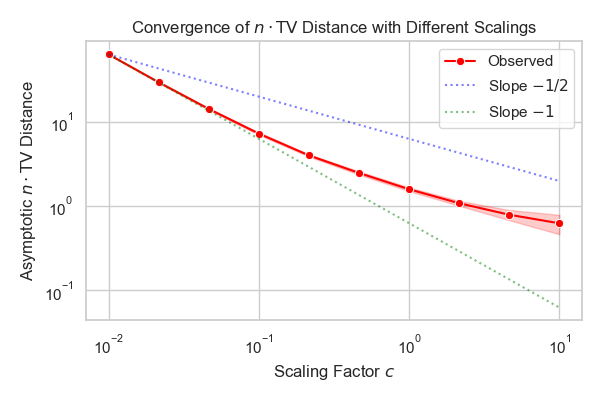}
        \caption{$p(0) \neq p^*$}
    \end{subfigure}

    \caption{Log-log plots of $Nd_N$ as a function of the scaling factor $c$ for (a) uniform $\beta = [0.5, 0.5, 0.5, 0.5]$ and (b) non-uniform $\beta \propto [1.1, 0.8, 0.9, 1.2]$. In both cases $p(0)$ is the uniform distribution. Dotted lines show the expected $\Theta(c^{-1})$ and $\Theta(c^{-1/2})$ scalings in the $p(0) \neq p^*$ and $p(0) = p^*$ regimes respectively.}
    \label{fig:convergence_plots}
\end{figure}

\subsection{Predictability}

In this section, we explore how the parameter $\beta$ affects the predictability of the sampler, as described in Section \ref{sec:adversarial_interpretation}. There, we showed that the key factor was whether or not $p^*$ has a unique maximizer. If it does not, then the predictability scales like $\Theta(\sqrt{\beta})$ above the naive baseline $p_{max}^*$ of IID sampling. If the maximizer is unique, we saw this first order correction vanishes, and the gap is exponentially small. We compare it to our two naive baselines: IID sampling and deterministic sampling. 

In IID sampling, the adversary will clearly always predict the index $i$ which maximizes $p_i^*$, and hence its long-term accuracy will be $p_{max}^*$. In greedy sampling, we can make the sampler a bit more unpredictable by sampling the next index uniformly from $\arg\max_{i} p_i(n)$, rather than just the purely deterministically. In the uniform $\beta$ case, this gives an accuracy of approximate $\frac{\log(d)}{d}$, compared to the $d^{-1}$ of IID sampling. In the non-uniform $\beta$ case, we estimate the error rate under the greedy model through a simple Monte-carlo simulation. Namely, assume at step $n$ the adversary has probability $m_n^{-1}$ of correctly guessing the next sample, where
    \[m_n = |\arg\max_{i} p_i(n)|,\]
is the number of outcomes that are maximally likely to be sampled. We generate samples until the average accuracy $N^{-1}\sum_{n=1}^N m_n^{-1}$ approximately converges. Similarly, we can likewise estimate the adversarial accuracy for the self-balancing sampler by averaging these estimates over many sampling trajectories.

The results are shown in Figure \ref{fig:predictability_plots}. In the uniform $\beta$ case, prediction accuracy increases gradually as the bias scale (c) increases, and in the small-(c) regime the gap above the IID baseline roughly follows the predicted $\Theta(\sqrt{c})$ scaling. In the non-uniform, unique-maximizer case, prediction accuracy increases more sharply as (c) grows. Conversely, as $c\downarrow 0$, the gap above the IID baseline decays much faster than the $\sqrt c$ scaling observed in the uniform case.

\begin{figure}[t]
    \centering
    \begin{subfigure}{\textwidth}
        \centering
        \includegraphics[width=\textwidth]{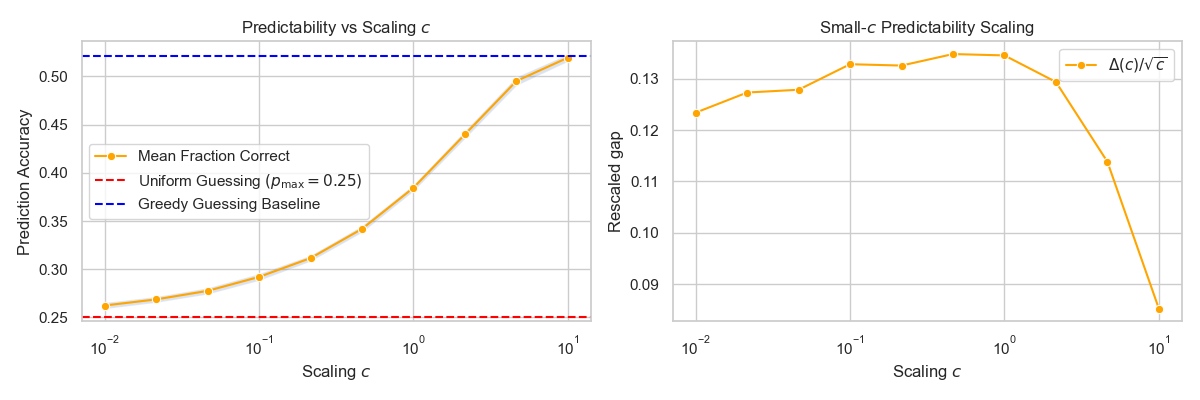}
        \caption{Uniform $\beta = [0.5, 0.5, 0.5, 0.5]$.}
        \label{fig:predictability_uniform_top}
    \end{subfigure}

    \vspace{0.3cm}

    \begin{subfigure}{\textwidth}
        \centering
        \includegraphics[width=\textwidth]{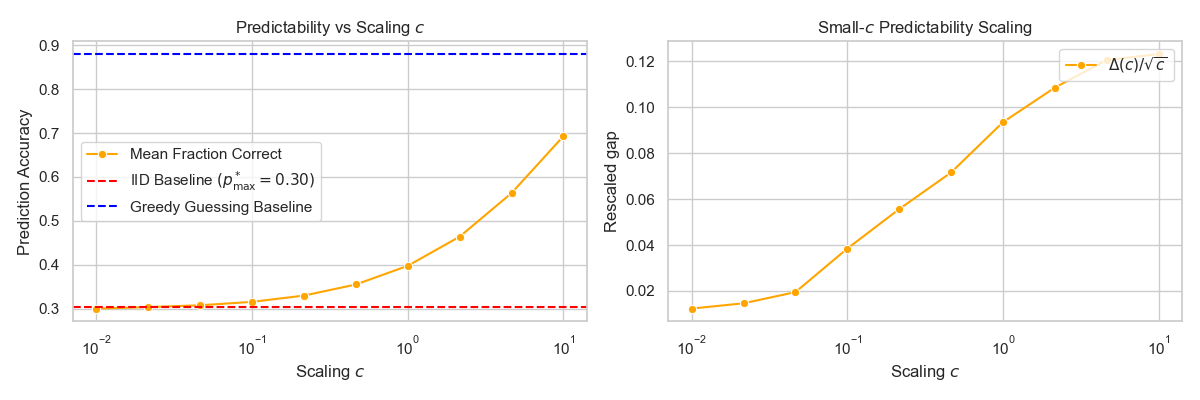}
        \caption{Non-uniform, distinct maximizer $\beta \propto [1.1, 0.8, 0.9, 1.2]$.}
        \label{fig:predictability_nonuniform_bottom}
    \end{subfigure}

    \caption{Predictability plots for different choices of $\beta$. (left) Relationship between prediction accuracy and the scaling $c$ of $\beta$. (right) The ratio of the gap $\Delta(c) := \text{accuracy} - p_{max}^*$ from the uniform baseline and the first order $\sqrt{c}$ expected scaling. }
    \label{fig:predictability_plots}
\end{figure}

\subsection{Trade-Off}

Lastly, we study the trade-off between predictability and the convergence rate. As discussed in the previous two sections, we quantify the former using the adversary's prediction accuracy, and the latter using the limit $\lim_{n\to \infty} n d_n$. Figure \ref{fig:trade_off} highlights the results.  

Note the predictability starts essentially at the baseline $p_{max}^*$, and increases as we increase the scaling $c$. It starts smaller in the uniform $\beta$ case merely because $p_{max}^*$ is smaller than in the non-uniform case. However, we see that in the latter the predictability is much slower to increase initially, but then quickly becomes more predictable as we scale $\beta$.

\begin{figure}[t]
    \centering

    \begin{subfigure}[t]{0.48\textwidth}
        \centering
        \includegraphics[width=\textwidth]{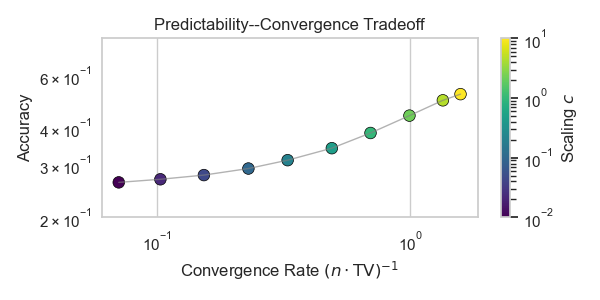}
        \caption{Uniform $\beta$.}
    \end{subfigure}
    \hfill
    \begin{subfigure}[t]{0.48\textwidth}
        \centering
        \includegraphics[width=\textwidth]{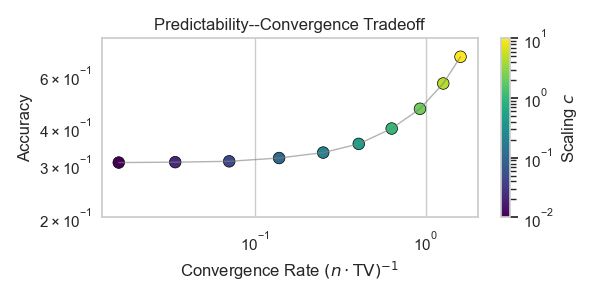}
        \caption{Non-uniform $\beta$.}
    \end{subfigure}

    \caption{Log-log plots of the trade-off between adversarial accuracy and the convergence rate $(Nd_N)^{-1}$.  (a) uniform $\beta = [0.5, 0.5, 0.5, 0.5]$ and (b) non-uniform $\beta \propto [1.1, 0.8, 0.9, 1.2]$. }
    \label{fig:trade_off}
\end{figure}

\subsection{Citizen's Assembly}
\label{sec:citizens_assembly}

In this section, we simulate the self-balancing sampler on the citizen's assembly problem discussed in Section~13 of~\cite{flanigan2021fair}. The setting is as follows.

Suppose we want to select an assembly of $k = 200$ people that includes at least $99$ members of each of the following categories: women, men, liberals, and conservatives. Let the pool consist of $1{,}000$ conservative men, $999$ liberal women, and $1$ conservative woman. The goal is to generate an assembly satisfying these quotas while ensuring that each individual in the pool has a reasonable chance of being selected. More precisely, we want to maximize the minimum individual inclusion probability induced by the sampling scheme. We refer to this minimum inclusion probability as the \textit{representativeness} of the scheme.

In this constructed population, the stratified sampling algorithm that selects $100$ uniformly drawn women and $100$ uniformly drawn men satisfies the quotas and selects each pool member with equal probability $10\%$. On the other hand, the \textit{LEGACY} algorithm discussed in \cite{flanigan2021fair} has only a $0.2\%$ chance of selecting the conservative woman. Thus, although LEGACY aims to balance covariates across the assembly, it provides poor ex ante representation for the rare class, since the conservative woman is almost never selected.

This stratified sampling strategy is tailored to the special structure of the example and does not provide a general solution. It works here because the population can be partitioned by gender into two groups of equal size, so selecting $100$ women and $100$ men simultaneously satisfies the gender quotas and gives every individual inclusion probability $100/1000 = 10\%$. In a more general population, the quota constraints may involve several overlapping attributes, and no single stratification need both satisfy the quotas and equalize individual inclusion probabilities. Our goal is therefore to test whether the self-balancing sampler can approximate the same effect in a more flexible way: biasing selection toward underrepresented attributes while preserving reasonable individual-level representativeness.

To apply the self-balancing sampler to this setting, we consider two sequential strategies:

\begin{enumerate}
    \item \textbf{Stratified self-balancing over classes}: We first choose a class in
    \[
        \{\text{man}, \text{woman}\} \times \{\text{conservative}, \text{liberal}\}
    \]
    using the self-balancing sampler, and then choose an individual uniformly at random from the selected class. The relative weights $(\beta_i^{-1})_{i=1}^4$ are chosen proportional to the initial class sizes. This ensures that, in the long-run, individuals are selected uniformly from the total population.

    \item \textbf{Attribute-balanced sampler}: Each individual is assigned weight proportional to
    \[
        \exp\left(
            -\beta_{\mathrm{gender}} N_{\mathrm{gender}}
            -\beta_{\mathrm{party}} N_{\mathrm{party}}
        \right).
    \]
    Equivalently, the sampler biases selection according to the current gender and party imbalances, and then samples among individuals with probabilities induced by these attribute-level weights.
\end{enumerate}

By design, the attribute-based sampler aims to match the marginal distributions of the assembly to those of the population. This tends to make it better at matching desired marginal quotas, but at the cost of fairness: it tends to under-sample rare joint classes (e.g. the conservative woman) and biases the sampling distribution away from the population distribution. On the other hand, the stratified sampler aims to be as fair as possible. In settings where fairness is the main objective, we recommend the stratified sampler.

As a baseline, we additionally study the simple strategy where individuals are selected uniformly at random to be part of the assembly. We call this the \textit{IID uniform baseline}. For each of these sampling strategies, we study the following quantities:

\begin{itemize}
    \item The probability that the quota constraints are satisfied.
    \item The individual inclusion probabilities induced by the sampling scheme, especially the worst-case inclusion probability.

\end{itemize}

 Moreover, we are interested in the influence the bias $\beta$ has on these quantities. To investigate this tradeoff, we simulate many assemblies sampled \textit{without replacement} according to these three sampling schemes. While our theory is developed for sampling with replacement, the difference between these situations for a large population size makes this a reasonable exploratory application. As before, we fix baseline bias parameters $\beta$ and rerun the citizen's assembly experiment with rescaled parameters $\beta' = c\beta$ over a range of scalings $c$.

For each value of $c$ and each sampler, we estimate both the probability that the quota constraints are met and the worst-case probability that a specific individual is selected. For the former, we use the fraction of simulated assemblies satisfying the quota constraints. For the latter, we estimate the inclusion probability $\hat p_i$ of each individual using the fraction of simulated assemblies in which that individual appears, and then take the minimum over individuals. To generate confidence bands, we bootstrap resample from the collection of simulated assemblies.

The estimator of the minimum selection probability is downward biased, since it is the minimum of many noisy empirical estimates. Heuristically, uniformly sampling a fraction $p$ of a population of size $m$ across $n$ assemblies, the empirical selection probability of individual $i$ satisfies
\[
    \hat p_i \sim n^{-1}\mathrm{Binom}(n,p) \approx N\left(p, \frac{p(1-p)}{n}\right),
\]
 for large $n$, and thus
\[
    \mathbb{E}\min_{i \leq m} \hat p_i
    \approx
    p + \mathbb{E}\min_{i \leq m}
    N\left(0, \frac{p(1-p)}{n}\right)
    \approx
    p - \sqrt{\frac{2p(1-p)\log m}{n}}.
\]
Thus, the empirical minimum underestimates the population-level inclusion probability by an amount of order $\sqrt{\log(m) / n}$. This motivates the $IID$ baseline.

The results are shown in Figure~\ref{fig:citizen_assembly} for the class sizes described above, using $10^4$ simulated assemblies and $1000$ bootstrap samples. Figure~\ref{fig:citizen_quota_prob} shows that quota-satisfaction probability increases with the bias scaling $c$ and is already close to one for relatively small bias. Figure~\ref{fig:citizen_representativeness} shows a sharper distinction between the two self-balancing variants: the stratified sampler remains essentially optimally representative across the tested range, whereas the attribute-balanced sampler loses representativeness at moderate bias. This appears to be a parity effect. Because the population consists almost entirely of conservative men and liberal women, the attribute-based sampler tends to alternate between these two classes as $\beta \to \infty$, while the conservative woman receives little or no mass. The simulations confirm that the conservative woman is again underrepresented in this high-bias regime.

\begin{figure}[t]
    \centering

    \begin{subfigure}[t]{0.48\textwidth}
        \centering
        \includegraphics[width=\textwidth]{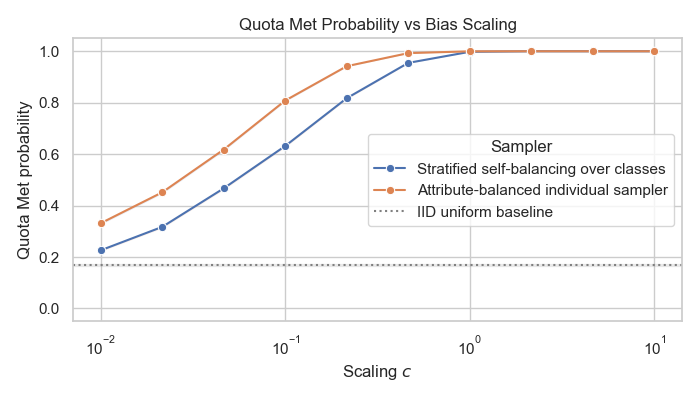}
        \caption{Quota satisfaction probability}
        \label{fig:citizen_quota_prob}
    \end{subfigure}
    \hfill
    \begin{subfigure}[t]{0.48\textwidth}
        \centering
        \includegraphics[width=\textwidth]{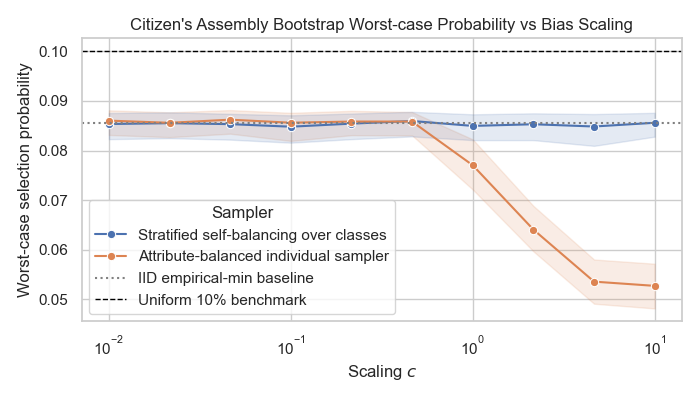}
        \caption{Worst-case inclusion probability}
        \label{fig:citizen_representativeness}
    \end{subfigure}

    \caption{Citizen's assembly simulation results. The effect of the rescaled bias parameter $\beta' = c\beta$ on the probability that the quota constraints are met and on the worst-case individual inclusion probability.}
    \label{fig:citizen_assembly}
\end{figure}


\newpage 
\section{Proof Appendix}
\label{appendix:proofs}

In this Appendix, we collect the formal proofs of all the main results of the paper.

\subsection{Section 2.2 Proofs: Potential-Based Sampling}
\label{appendix:2_2_proofs_potential_based_sampling}

\proofappendixheading{Proof of Proposition~\ref{prop:projection_centered_counts_are_markovian}}
\proofappendixstatement{Proposition}{If $\pi$ is $v$-invariant and 
$\mathrm{gcd}(v_1,..., v_d) = 1$ then $v_dN_i(n) - v_i N_d(n),$ is a Markov chain on $\Z^{d-1}$.}

    \begin{proof}
        Let $x,y \in \N^d$. Define
            \[T_v(x) := (v_dx_i - v_ix_d)_{i=1}^{d-1}.\]
        \noindent First we show that $x \sim_v y$ if and only if $T_v(x) = T_v(y)$. If $y = x + kv$ then
            \[T_v(y) = v_dy_i - v_iy_d = (v_dx_i -v_ix_d) + k(v_dv_i - v_iv_d) = v_dx_i - v_ix_d = T_v(x),\]
        \noindent as desired. Conversely, if
        $T_v(x) = T_v(y)$ then re-arranging shows
            \[\frac{x_i - y_i}{v_i} = \frac{x_d - y_d}{v_d}, \quad \forall i \in [d-1].\]
        \noindent Since the right hand side is independent of $i$, we get $x = y + kv$ for some constant $k \in \R$. Since $x,y\in \N^d$ and $\mathrm{gcd}(v_1,..., v_d) = 1$, we must in fact have $k \in \Z$. Thus $x \sim_v y$. 

        Now we show that the counts process descends to a Markov chain under $T_v$. If $T_v(x)=T_v(y)$, then $y=x+kv$ for some $k\in\mathbb Z$, and hence $v$-invariance gives $\pi(y)=\pi(x)$. Therefore, the conditional law of $T_v(N(n+1))$ given $N(n)=x$ depends on $x$ only through $T_v(x)$. Thus $(T_v(N(n)))_{n\geq 0}$ is a Markov chain on $T_v(\mathbb N^d)$.
    \end{proof}

\proofappendixheading{Proof of Proposition~\ref{prop:charactizing_v_invariance}}
\proofappendixstatement{Proposition}{If $\pi^\phi$ is $v$-invariant, then for each $i$ there exist constants $b>0$ and $c_{i,k}$ such that $\phi_i(mv_i+k)=c_{i,k}b^m$ for all $m\in\N$ and $k\in\{0,\dots,v_i-1\}$. If, in addition, each $\phi_i$ is log-convex or log-concave, then $\phi_i(n)=c_ir_i^n$ and $(\log r_1,\dots,\log r_d)\propto (v_1^{-1},\dots,v_d^{-1})$.}
\begin{proof}
    Denote $S(x) := \sum_j \phi_j(x_j)$. 
    Then $v$-invariance implies:
    \[\frac{\phi_i(x_i)}{S(x)} = \frac{\phi_i(x_i + v_i)}{S(x + v)}, \quad \frac{\phi_j(x_j)}{S(x)} = \frac{\phi_j(x_j + v_j)}{S(x + v)}.\]
    \noindent Dividing these equalities and rearranging yields:
    \[\frac{\phi_i(x_i + v_i)}{\phi_i(x_i )} = \frac{\phi_j(x_j + v_j)}{\phi_j(x_j)}.\]
    \noindent As the left side is independent of $x_j$ and the right side is independent of $x_i$, these ratios must equal some constant $b > 0$ independent of $x_i, x_j$. 
    Hence for every $i \in [d]$ and $k \in \{0, 1, ..., v_i-1\}$ we must have
    \[\phi_i(m v_i + k) = \phi_i(k) \cdot b^m, \quad \forall m \in \N.\]
    \noindent Log-concavity implies for $k \in \{0, 1, ..., v_{i} - 2\}$:
    \[\frac{\phi_i(0)}{\phi_i(1)} \leq \frac{\phi_i(k)}{\phi_i(k+1)} \leq \frac{\phi_i(v_i)}{\phi_i(v_i+1)} = \frac{\phi_i(0)b}{\phi_i(1)b}.\]
    \noindent If instead $\phi_i$ is log-convex, the direction of the inequalities all swap. Either way, we see these ratios are constant, and thus $\phi_i(m) = c_i\alpha_i^m$ where $\alpha_i = b^{1/v_i}$. 
    Thus we have
        \[\log(\alpha_i) = \log(b) \cdot v_i^{-1},\]
    \noindent giving the desired result.
\end{proof}

\proofappendixheading{Proof of Lemma~\ref{lem:projection_centered_has_stationarity}}
\proofappendixstatement{Lemma}{
       For $\beta_i = cv_i^{-1}$ for $v \in \N^d$ and $c > 0$, the projection-centered chain $\tilde{Y}^v(n) = v_dN_i(n) - v_i N_d(n)$ associated to the self-balancing sampler~\eqref{eqn:self_balancing_sampler_transition_probs} is irreducible and 
    positive recurrent, and thus has a unique stationary distribution $\eta$ on some subset of $\Z^{d-1}$.}

    \begin{proof}
        Clearly the state-space of $\tilde{Y}^v$ is
        the subset of $\Z^{d-1}$ generated by nonnegative integer combinations of the vectors $\{v_d e_i\}_{i=1}^{d-1} \cup \{-v\}$.
        Of course, this does not always span $Z^{d-1}$, even when $\mathrm{gcd}(v_i) = 1$. For example, if $v_1 = 2, v_2 = 3, v_3 = 4$ then the first coordinate
        of $\tilde{Y}^v$ is always a multiple of $2$. 

        Nonetheless, we show the chain is still irreducible on this state-space. To see why, suppose:
        \[x = v_d\sum_{i=1}^{d-1} a_i e_i - a_d v, \quad a_i \geq 0.\]
        \noindent Clearly we can go from $0$ to $x$ by sampling $a_i$ times from index $i$ for each $i \in [d-1]$ and sampling $a_d$ times 
        from index $d$. Conversely, we can go from $x$ to $0$ by sampling index $i \in [d-1]$ a total of $k v_i - a_i$ times, for some $k$ large enough so $k v_i \geq a_i$ for all $i \in [d - 1]$ and $kv_d \geq a_d$. 
        Then sample coordinate $d$ a total of $k v_d - a_d$ times. 

        Lastly, we show the chain is positive recurrent by developing a suitable Lyapunov function and proving a Foster-Lyapunov drift condition. See, for example, \cite{meyn2009markov}, for more information about this technique.

        Let $\beta_i = v_i^{-1}$ and define
        \[Q_i := \beta_i N_i - \log(p_i(0) \beta_i), \quad \overline{Q} = \frac{1}{d} \sum_{i=1}^{d} Q_i.\]
        \noindent The constants $\log(p_i(0) \beta_i)$ serve an analogous role to $c_\beta$ in our diffusive limit:
        the shift the chain so it is centered at the zero-drift location. Note
        \[p_i = \frac{p_i(0)e^{-\beta_i N_i}}{\sum_j p_j(0)e^{-\beta_j N_j}} = \frac{v_i e^{-Q_i}}{\sum_j v_j e^{-Q_j}},\]
        \noindent is a function of $Q$. Using this, define Lyapunov function
        \[V(Q) := \sum_{i=1}^{d} Z_i^2, \quad Z_i := Q_i - \overline{Q}.\]
        \noindent Then $V$ is a function of the projection-centered counts $v_d N_i - v_i N_d$ alone. Moreover, $V(Q) = V(Q + c 1)$ for all $c \in \R$, so
        it suffices to consider the case $Q_d = 0$. 
        We will show $V$ has negative drift outside some finite set. First, note as $\sum Z_i = 0$
        \[V(Q + \Delta Q) = 2 \sum_{i=1}^{d} Z_i \Delta Q_i + \sum_{i=1}^{d} (\Delta Q_i - \overline{\Delta Q})^2\]
        \noindent Thus $V$ has expected drift
        \[\E[ \Delta V(Q) \mid Q = q] = 2 \sum_{i=1}^d Z_i p_i(q)\beta_i + \frac{d-1}{d} \sum_{i=1}^d p_i(q) \beta_i^2\]
        \noindent The second term is at most $C = \max \beta_i^2$. The first term is sufficiently negative on
        \[F := \left\{Q \in \N^d : \max_i Q_i - \min_i Q_i > R, \quad Q_d = 0\right\}.\]
        \noindent To see why, suppose $M = \arg\max_i Q_i$ and $m = \arg\min_i Q_i$.
        \begin{align*}
            \sum_{i=1}^d Z_i p_i\beta_i & = \frac{1}{d} \sum_{i < j} (q_i - q_j)(p_i \beta_i - p_j\beta_j).\\
            & = \frac{1}{d \sum_j v_j e^{-q_j}} \sum_{i < j} (q_i - q_j)(e^{-q_i} -  e^{-q_j}).\\
            & \leq \frac{(q_m - q_M)(e^{-q_m} -  e^{-q_M})}{d \sum_j v_j e^{-q_j}} \\
            & \leq \frac{(q_m - q_M)(e^{-q_m} -  e^{-q_M})}{d^2 e^{-q_m} \max_i v_i} \\
            & \leq - \frac{R(1-e^{-R})}{d^2 \max_i v_i} \cdot \\
        \end{align*}

        \noindent where in the first inequality we observed every term in the sum is nonpositive. The complement $F^c$ is a finite subset of $\Z^{d-1}$, 
        so we have the desired drift condition. 
    \end{proof}

\proofappendixheading{Proof of Lemma~\ref{lem:sampling_at_stationarity_gives_pstar}}
\proofappendixstatement{Lemma}{Suppose $\pi$ is $v$-invariant and the associated projection-centered chain has stationary distribution $\eta$. If $y\sim\eta$ and $X\mid y\sim\pi(y)$, then $\P(X=i)=\frac{v_i}{\sum_j v_j}$ for each $i\in[d]$.}
\begin{proof}
    Suppose $\tilde{Y}^v \sim \eta$. Let $p_i := \P(X = i)$ and define $\Delta_n \tilde{Y}^v := \tilde{Y}^v(n+1) - \tilde{Y}^v(n)$. If $\tilde{Y}^v(n) \sim \eta$, then 
    $\tilde{Y}^v(n+1) \sim \eta$ as well. Hence $\E \Delta_n \tilde{Y}^v = 0 \in \Z^{d-1}$. As discussed above, we can view $\tilde{Y}^v(n)$ as the image of $N(n)$ under the map
    \[x \mapsto (v_dx_i - v_ix_d)_{i=1}^{d-1}.\]
    \noindent Hence:
    \[(\Delta_n \tilde{Y}^v)_i = \begin{cases}
        v_d & w.p. \quad p_i,\\
        -v_i & w.p. \quad p_d,\\
        0 & w.p. \quad 1 - p_i - p_d.
    \end{cases}\]
    \noindent Hence the $i^{th}$ coordinate of this expectation is:
    \[(\E \Delta_n \tilde{Y}^v)_i = v_dp_i - v_ip_d = 0.\]
    \noindent Thus $p_i = \frac{v_i}{v_d}p_d$. Since $\sum_{i=1}^d p_i = 1$, we see:
    \[\sum_{i=1}^d \frac{v_i}{v_d} = \frac{1}{p_d} \to p_d = \frac{v_d}{\sum_{j=1}^d v_j},\]
    \noindent which yields the desired result. 
\end{proof}

\newpage 
\subsection{Section 2.3 Proofs: Convergence Rates}
\label{appendix:2_3_convergence_rate_proofs}


\proofappendixheading{Proof of Proposition~\ref{prop:empirical_averages}}
\proofappendixstatement{Proposition}{For every sequential sampling rule,
\[
\E\left[\max_{1\le i\le d}\left|\mu_i(n)-\frac1n\sum_{m=0}^{n-1}p_i(m)\right|\right]\lesssim \frac{d}{\sqrt n}.
\]
If the projection-centered count process is irreducible and positive recurrent, then both $\mu_i(n)$ and $\frac1n\sum_{m=0}^{n-1}p_i(m)$ converge almost surely.}

\begin{proof}
    For a given $1 \leq i \leq d$, consider the sequence $M_{i}(0) = 0$, $M_{i}(n) \coloneqq N_{i}(n) - \sum_{m=0}^{n-1} p_{i}(m)$. This is a martingale adapted to $\mathcal F_n \coloneqq \sigma(X_1,\dots,X_n)$, as
    $$\mathbb E[M_{i}(n+1) - M_{i}(n) \mid \mathcal F_n] = \mathbb E[\mathbf 1_{\{X_{n+1} = i\}} - p_{i}(n) \mid \mathcal F_n] = p_{i}(n) - p_{i}(n) = 0.$$
    Moreover, these martingales each have bounded differences $|M_i(n+1) - M_i(n)| \leq 1$, so the Azuma-Hoeffding inequality gives
    $$\P(|M_i(n)| \geq n\varepsilon) \leq 2 e^{-n \varepsilon^2/2}.$$

    This yields exponential concentration for $\mu_i(n) - \frac{1}{n} \sum_{m=0}^{n-1} p_i(m) = M_i(n)/n$ around $0$, so by the Borel-Cantelli theorem, the difference converges almost surely. And if the count process for a sequential sampling rule is irreducible and positive recurrent, then the ergodic theorem for irreducible positive-recurrent countable-state Markov chains, applied to the projection-centered count process, guarantees almost sure convergence of $\frac{1}{n} \sum_{m=0}^{n-1} p_i(m)$.
    
    For the expectation bound, first, a union bound gives
    $$\P \left(\max_{1 \leq i  \leq d} \left |\mu_{i}(n) - \frac{1}{n} \sum_{m=0}^{n-1} p_{i}(m) \right| \geq \varepsilon \right) \leq 2d e^{-n\varepsilon^2/2}.$$
    We conclude that
    \begin{align*}
        \mathbb E \left[\max_{1 \leq i  \leq d} \left |\mu_{i}(n) - \frac{1}{n} \sum_{m=0}^{n-1} p_{i}(m) \right|  \right] &= \int_0^\infty \P \left(\max_{1 \leq i  \leq d} \left |\mu_{i}(n) - \frac{1}{n} \sum_{m=0}^{n-1} p_{i}(m) \right| \geq \varepsilon \right) \, d\varepsilon \\
        &\leq \int_0^\infty 2d e^{-n\varepsilon^2/2} \, d\varepsilon \\
        &= d \sqrt{\frac{2\pi }{n}}.\qedhere
    \end{align*}
\end{proof}

\proofappendixheading{Proof of Theorem~\ref{thm:average_case_conv_rate}}
\proofappendixstatement{Theorem}{The empirical distribution $\mu_n$ converges at a rate of $1/n$:
    $$\E\| \mu(n) - p^* \|_{\operatorname{TV}} = O(1/n).$$
    More granularly, for $\max_i \beta_i \leq 1$,
    $$\mathbb E \| \mu(n) - p^*\|_{\operatorname{TV}} \lesssim \begin{cases}
        \frac{1}{n}\sum_i \frac{1}{\beta_i} & \text{if $p_0 \neq p^*$,} \\
        \frac{1}{n}\sum_i \frac{1}{\sqrt{\beta_i}} & \text{if $p_0 = p^*$.}
    \end{cases}$$
    where the constant depends only on $p(0)$, $p^*$, $\kappa := \frac{\max_i \beta_i}{\min_i \beta_i}$, and $d$.}

\noindent The key idea is that
    \[\E ||\mu_{n} - p^*||_{TV} = \frac{1}{2n} \sum_i |N_i(n) - np_i^*|| = \frac{||Y(n)||_1}{2n}\]
\noindent is a function of the moments of the mean-centered counts $Y(n)$.

To control these moments, we will develop a Foster-Lyapunov drift condition for some suitable potential. This requires a quantitative version of Theorem 1 of \cite{PemantleRosenthal1999} which will allow us to bound these moments in terms of the level of drift. This is the content of the following Lemma. Since the drift of $Y(n)$ is controlled by $\beta$, this uncovers the desired $\beta$ dependence.

\begin{lemma} \label{lem:drift_lemma}
    Let $U_n \geq 0$ be $\cF_n$-adapted for $n \geq 0$. Suppose there exists $a,b,c > 0$ with $a \leq c$ so
    \begin{enumerate}
        \item[(i)] $\E[U_{n+1} - U_n \mid \mathcal F_n] \leq - a$ on $\{U_n \geq b\}$,
        \item[(ii)] $|U_{n+1} - U_n| \leq c$ a.s.
    \end{enumerate}
    Then
    $$\sup_{n \geq 0} \E[U_n] \leq \E[U_0] + b + c + \frac{c^2}{2a}.$$
\end{lemma}

\begin{proof}[Proof of Lemma~\ref{lem:drift_lemma}]
    The idea is to compare (a shifted version of) $U_n$ to the ``worst case'' birth-death process with properties (i) and (ii), which will satisfy the desired bound. To this end, define the shifted version $X_n := (U_n - b)_+$ and the following birth-death process which makes jumps of size $c$:
        $$Q_0 := \max \{ c,X_0\}, \qquad Q_{n+1} := \max \{ c, Q_n + Z_{n+1}\},$$
    where $Z_1,Z_2,\dots$ are iid with
        $$\P(Z_n = c) = \theta, \qquad \P(Z_n = -c) = 1-\theta, \qquad \theta := \frac{c-a}{2c} \in [0,1].$$
    \noindent Note $Z_i$ is chosen so $\E Z_i = -a$ and $|Z_i| \leq c$, so $(i)$ and $(ii)$ are satisfied by $Q_n$. Moreover, $Z_i$ essentially maximizes the variance among all such random variables. Since the downside of $Q_n$ is limited to $c$, this choice of increment in a sense maximizes the expectation of $Q_n$.

    To compare $X_n$ and $Q_n$, we claim for any increasing, convex $F$
    \begin{equation}
        \E[F(X_{n+1}) \mid \mathcal F_n] \leq \theta F(X_n + c) + (1-\theta)  F(\max \{c, X_n - c \}) = \E[F(Q_{n+1}) \mid \mathcal F_n].
        \label{claim:operator_T_increasing_convex_order}
    \end{equation}
    \noindent In words, this says $(Q_{n+1} \mid \cF_n)$ stochastically dominates $(X_{n+1} \mid \cF_n)$ in the order $\leq_{icx}$ induced by increasing, convex functions. Define an operator $T$ by
        \[(TF)(x) := \theta F(x + c) + (1-\theta)F(\max \{ c,x-c\}),\]
    \noindent so that $TF(Q_n) = \E[F(Q_{n+1}) \mid \mathcal F_n]$.

    To prove the claim, first note that if $X_n = 0$, then the inequality follows from (ii) as
    $$F(X_{n+1})  = F((U_{n+1} - b)_+)\leq F(c).$$
    When $X_n > 0$, let $G(z) := F(\max \{ c,X_n +z\})$. Then $G$ is increasing and convex, which gives
    $$G(z) \leq \frac{c+z}{2c}G(c) + \frac{c-z}{2c} G(-c), \qquad \forall |z| \leq c.$$
    In particular, we know that $|U_{n+1} - U_n| \leq c$ almost surely, so applying this inequality and taking conditional expectations gives
    \begin{align*}
        \E[G(U_{n+1} - U_n) \mid \mathcal F_n] &\leq \frac{c + \E[U_{n+1} - U_n \mid \mathcal F_n]}{2c}(G(c) - G(-c)) + G(-c).
        \intertext{Since $G$ is increasing, $G(c) - G(-c) \geq 0$. We also know that $\E[U_{n+1} - U_n \mid \mathcal F_n] \leq -a$, so}
        &\leq \theta (G(c) - G(-c)) + G(-c) \\
        &= \theta G(c) + (1-\theta) G(-c) \\
        &= \theta F(X_n + c) + (1-\theta) F(\max \{ c,X_n - c\}).
    \end{align*}
    Now, since $X_n > 0$, we have $X_n = U_n -b$, so that
    $$G(U_{n+1} - U_n) = F( \max \{ c, X_n + U_{n+1} - U_n\}) = F( \max \{ c, U_{n+1} - b\}) \geq F(X_{n+1})$$
    by the monotonicity of $F$; this proves claim~\eqref{claim:operator_T_increasing_convex_order}.

     Taking expectations of both sides in the claim inequality gives $\E[F(X_{n+1})] \leq \E[TF(X_n)]$. Since $T$ maps the set of increasing convex functions onto itself, we can apply claim~\eqref{claim:operator_T_increasing_convex_order} iteratively starting with $F(x) = x$ to get
    \begin{equation*}
        \E[X_n] = \E[F(X_n)] \leq \E[TF(X_{n-1})] \leq ... \leq \E[T^nF(X_0)].
    \end{equation*}
    \noindent Since $T^nF$ is increasing and $X_0 \leq Q_0$ we see
        \[\leq \E T^n F(Q_0) = \E Q_n.\]
    It now only remains to bound $\sup_n \E[Q_n]$. If we couple $Q_n$ with $Q_n'$, a version started at $Q_0' = c$ but with the same increments, then we get
    $$\E[Q_n] \leq \E[Q_n'] + \E[Q_0 - c].$$
    The stationary distribution of the $Q_n'$ chain is
    $$\pi(kc) := (1-r)r^{k-1}, \qquad r := \frac{\theta}{1-\theta} = \frac{c-a}{c+a},$$
    and the mean of the chain at stationarity is therefore
    \begin{equation*}
        \sum_{k=1}^\infty ck(1-r)r^{k-1} = \frac{c}{1-r} = \frac{c}{2} + \frac{c^2}{2a}.
    \end{equation*}
    Since $Q_0'$ is less than a version of the chain started at stationarity and coupled to have the same increments, we get
    $$\E[Q_n'] \leq \frac{c}{2} + \frac{c^2}{2a}.$$
    So, in total, we get
    \begin{align*}
        \E[Q_n] &\leq \E[Q_n'] + \E[Q_0 - c] \leq \E[Q_0] - \frac{c}{2} + \frac{c^2}{2a} \leq \E[Q_0] + \frac{c^2}{2a}.
    \end{align*}
    Unwrapping the definitions gives the final bound
    $$\E[U_n] \leq b + \E[(U_n - b)_+] \leq b + \E[\max\{c,U_0-b\}] + \frac{c^2}{2a} \leq \E[U_0] + b + c + \frac{c^2}{2a},$$
    uniformly in $n$.
\end{proof}

\begin{proof}[Proof of Theorem \ref{thm:average_case_conv_rate}]
    We treat the general case by essentially applying shifts to all of the counts to compensate for the ``wrong'' initial distribution. In particular, let
        $$r_i := \log \frac{p_i(0)}{p_i^*}, \qquad \overline r := \sum_j p_j^* r_j, \qquad c_i := r_i - \overline r, \qquad \xi_i := \beta_i Y_i - c_i.$$
    Then
        $$\| Y(n) \|_1 = \sum_i \left| \frac{c_i + \xi_i(n)}{\beta_i}\right| \leq \sum_i \frac{|c_i|}{\beta_i} + \sum_i \frac{|\xi_i(n)|}{\beta_i}.$$
    When $p_i(0) = p_i^*$, the $c_i$ terms disappear, so it suffices to show that $\E[|\xi_i(n)|] \lesssim \sqrt{\beta_i}$ when $p(0) = p^*$ and $\lesssim 1$ otherwise.

    To handle all the $\xi_i(n)$ at once, consider $V_n := \sum_i \frac{\xi_i^2}{\beta_i}$; we will control $\E[V_n]$ via a drift bound. If $X_{n+1} = k$, then
    $$\Delta \xi_i(n) := \xi_i(n+1) - \xi_i(n) = \beta_i( \mathbf 1_{\{i = k\}} - p_i^*),$$
    so we may compute (keeping in mind that $\sum_i p_i^* \xi_i(n) = 0$ for all $n$)
        \[V_{n+1} - V_n =\sum_{i} \frac{\Delta \xi_i(n) \left(\xi_i(n+1) + \xi_i(n)\right)}{\beta_i} = 2\xi_{X_{n+1}}(n) + \beta_{X_{n+1}}(1-p_{X_{n+1}}^*).\]
    \noindent and hence
    \begin{align*}
        \E[V_{n+1} - V_n \mid N(n)] &= \sum_{i} p_i(n) \left[2\xi_i(n) + \beta_i(1-p_i^*)\right]\\[5pt]
        & \leq 2\sum_i  \xi_i(n) p_i(n) + \sum_i \beta_i p_i(n)\\[5pt]
        & \leq 2\sum_i \xi_i(n) (p_i(n) - p_i^*) + \max_i \beta_i.
    \end{align*}

    To upper bound the first term in terms of (negative) $V_n$, recognize that it is an expectation with respect to a $p_i^*$, tilted in the direction of $-\xi$; in particular, $p_i(n) = \frac{p_i^* e^{-\xi_i(n)}}{\sum_j p_j^* e^{-\xi_j(n)}}$, and the term is $\E[\Xi]$, where $\Xi = \xi_i$ with probability $p_i(n)$. The more we tilt the distribution in the direction of $-\xi$, the lower the expectation of $\Xi$ will be. To make this intuition precise, define the following interpolation for $0 \leq t \leq 1$:
        $$Z(t) := \sum_i p_i^* e^{-t \xi_i}, \qquad q_i(t) := \frac{p_i^* e^{-t \xi_i}}{Z(t)}, \qquad f(t) := \log Z(t).$$
    Then $q(0) = p^*$, $q(1) = p(n)$, and if $\Xi(t) = \xi_i$ with probability $q(t)$ we have
        $$f'(t) = - \sum_i \xi_i q_i(t) = -\E[\Xi(t)], \qquad f''(t) = \Var(\Xi(t)) \geq 0.$$
    So we can calculate the desired expectation by tracking this interpolation:
        $$\sum_i \xi_i (p_i(n)-p_i^*) = - (f'(1) - f'(0)) = -\int_0^1 f''(t) dt \leq -\int_0^{\min\{1,1/\|\xi(n)\|_\infty\}} f''(t) \, dt.$$
    For $0 \leq t \leq \min\{1,1/\|\xi(n)\|_\infty\}$, $p_i^*$ and $q_i(t)$ are readily comparable:
        $$\frac{1}{e} \leq Z(t) \leq e,$$
    which gives
        $$\frac{p_i^*}{e^2} \leq q_i(t) \leq e^2 p_i^*.$$
     For such $t$, by the variational characterization of variance
    \begin{align*}
        f''(t) &= \min_a \sum_i q_i(t) (\xi_i(n)- a)^2 \\
        &\geq \frac{1}{e^2} \min_a \sum_i p_i^* (\xi_i(n)-a)^2 \\
        &= \frac{1}{e^2} \Var(\Xi(0)) \\
        &= \frac{1}{e^2} \sum_i \xi_i^2 p_i^* \\
        &= \frac{1}{e^2} S_\beta^{-1} V_n.
    \end{align*}
    for $S_\beta = \sum \beta_i^{-1}$. Putting it all together, we get a bound of
    \begin{align*}
        \E[V_{n+1} - V_n \mid N(n)] &\leq -\frac{2}{e^2} \min \{ 1, 1/\|\xi(n)\|_\infty\}S_\beta^{-1} V_n +  \max_i \beta_i.
        \intertext{We know that $\|\xi(n)\|_\infty \leq \sqrt{V_n\max_i \beta_i}$, so this is}
        &\leq -\frac{2}{e^2} S_\beta^{-1} \min \left\{V_n, \sqrt {\frac{V_n}{\max_i \beta_i} } \right\} +  \max_i \beta_i.
        \intertext{We have $S_\beta = \sum_j \beta_j^{-1} \leq \frac{d}{ \min_i \beta_i} = d \kappa \frac{1}{\max_i \beta_i} $, so that for $\max_i \beta_i \leq 1$,}
        &\leq -\frac{2}{e^2d \kappa} (\max_i \beta_i) \min \{V_n, \sqrt {V_n} \} + \max_i \beta_i.
    \end{align*}
    Since $\sqrt V_n \leq V_n$ for $V_n \geq 1$, we will actually establish our drift bound in terms of $U_n := \sqrt{1+V_n}$. Notice that by the tangent-line inequality for the concave function $x \mapsto \sqrt{1+x}$,
    \begin{align*}
        \E[U_{n+1} - U_n \mid N(n)] &\leq \frac{\E[V_{n+1} - V_n \mid N(n)]}{2 \sqrt{1+V_n}} \\
        &\leq \frac{-\frac{2}{e^2d \kappa} (\max_i \beta_i) \min \{V_n, \sqrt {V_n} \} + \max_i \beta_i}{2\sqrt{1+ V_n}}.
        \intertext{For $U_n \geq \sqrt{1 + e^4d^2 \kappa^2}$, we have $V_n \geq 1$ and $\frac{2}{e^2d \kappa} \sqrt V_n \geq 2$, so we get}
        &\leq - \frac{\max_i \beta_i}{2\sqrt{2}e^2d \kappa}.
    \end{align*}
    We also need a bound on the jump size for $U_n$: The map $x \mapsto \sqrt{1+\|x\|_\beta^2}$ is 1-Lipschitz with respect to the norm $\|x\|_\beta := \sqrt{\sum_i x_i^2/\beta_i}$, so again as $\sum_i \xi_i p_i^* = 0$ we get
    \begin{align*}
        |U_{n+1} - U_n| &\leq \sqrt{\sum_i \frac{(\xi_i(n+1) - \xi_i(n))^2}{\beta_i}} \\
        &= \sqrt{\sum_i (\mathbf 1_{\{X_{n+1} = i\}} - p_i^*) \Delta \xi_i(n)} \\
        &\leq \sqrt{ \max_i \beta_i}.
    \end{align*}
    \noindent Applying Lemma~\ref{lem:drift_lemma} with
        \[a=\frac{\max_i\beta_i}{2\sqrt{2}e^2d\kappa}, \qquad b=\sqrt{1+e^4d^2\kappa^2},\qquad c=\sqrt{\max_i\beta_i},\]
    \noindent and using $\xi_i(0) = -c_i$, we obtain
    \begin{align*}
        \sup_n \E[\sqrt{V_n}]
        &\leq \sup_n \E[U_n] \\
        &\leq \E[U_0]
            +\sqrt{1+e^4d^2\kappa^2}
            +\sqrt{\max_i\beta_i}
            +\sqrt{2}e^2d\kappa \\[5pt]
        &=
        \sqrt{1+\sum_j\frac{c_j^2}{\beta_j}}
            +\sqrt{1+e^4d^2\kappa^2}
            +\sqrt{\max_i\beta_i}
            +\sqrt{2}e^2d\kappa \\[5pt]
        &\leq
        \sqrt{1+\frac{\kappa}{\max_i\beta_i}
            \sum_j c_j^2}
            +\sqrt{1+e^4d^2\kappa^2}
            +1+\sqrt{2}e^2d\kappa \\
        &\lesssim
        \begin{cases}
            1,&\text{if }p(0)=p^*,\\[3pt]
            1/\sqrt{\max_i\beta_i},&\text{if }p(0)\neq p^*.
        \end{cases}
    \end{align*}
    where we have absorbed constants that depend only on $p(0),p^*,\kappa$, and $d$. Note, in particular, that the specific form of the constants in Lemma~\ref{lem:drift_lemma} makes the $\max_i \beta_i$ dependence cancel in the last term.
    
    We finish by applying this bound to $\xi_i$:
    \begin{align*}
        \E[|\xi_i(n)|] &= \sqrt{\beta_i} \E \left[ \frac{|\xi_i|}{\sqrt{\beta_i}} \right] \leq \sqrt{\beta_i} \E[\sqrt{V_i}] \lesssim \begin{cases}
        \sqrt{\beta_i} & \text{if $p(0) = p^*$} \\
        1 & \text{if $p(0) \neq p^*$}
        \end{cases}
    \end{align*}
    for every $i$.
\end{proof}

\proofappendixheading{Proof of Theorem~\ref{thm:time-dep-conv-rate} (Convergence Rates for Time-Inhomogeneous Biasing)}
\proofappendixstatement{Theorem}{If $\beta_i(n) = a(n)b_i$ for all $i$, where $a(n) \downarrow 0$ as $n \to \infty$ and $\sup_n a(n+1)^{-1} - a(n)^{-1} < \infty$, then
    $$\mathbb E \| \mu(n) - p^*\|_{\operatorname{TV}} \lesssim \begin{cases}
        \frac{1}{n} \sum_i \frac{1}{a(n)} & \text{if $p(0) \neq p^*$,} \\
        \frac{1}{n} \sum_i \frac{1}{\sqrt{a(n)}} & \text{if $p(0) = p^*$,}
    \end{cases}$$
    where the constant depends only on $p(0),d, \sup_n a(n+1)^{-1} - a(n)^{-1}$, and the $b_i$.}

\begin{lemma}[Tilting bound] \label{lem:tilting-bound}
    Let $q_i > 0$ with $\sum_i q_i = 1$. Then there exists a constant $c = c(q) > 0$ such that for every $x \in \R^d$ with $\sum_i q_i x_i = 0$,
    $$\frac{\sum_i q_i e^{-3tx_i/4}}{\sum_j q_j e^{-tx_j}} \leq e^{-c \min \{ t^2\sum_i q_i x_i^2,1 \}}$$
    for all $t > 0$.
\end{lemma}

\begin{proof}[Proof of Lemma~\ref{lem:tilting-bound}]
    Denote $\sigma^2 := \sum_i q_i x_i^2$, and let $y_i := x_i/\sigma$; we may rule out the case $\sigma = 0$ since the desired inequality holds with equality when all $x_i = 0$. Then $\sum_i q_i y_i = 0$ and $\sum_i q_i y_i^2 = 1$, and we want to show that
    $$\log \sum_i q_i e^{-t \sigma y_i} - \log \sum_i q_i e^{-3t \sigma y_i/4} \geq c \min \left \{ t^2 \sigma^2 , 1\right \}.$$
    Define $f(u) := \log \sum_i q_i e^{-u y_i}$, so we wish to show that
    $$f(u) - f(3u/4) \geq c \min \{ u^2,1\}.$$
    Observe that
    $$f(0) = 0, \qquad f'(0) = -\sum_i q_i y_i = 0, \qquad f''(0) = \sum_i q_i y_i^2 - \left(\sum_i q_i y_i \right)^2 = 1,$$
    and consider the set
    $$K_q:= \left\{ y \in \R^d: \sum_i q_i y_i = 0, \sum_i q_i y_i^2 = 1 \right \}.$$
    This set is compact: it is closed, and for every $y \in K_q$, $q_i y_i^2 \leq \sum_j q_j y_j^2 = 1$, so $|y_i| \leq q_i^{-1/2}$ for every $i$.

    Moreover, the third derivatives of $f$ are uniformly bounded for $y \in K_q$ and $u$ in a fixed neighborhood of zero. Therefore, uniformly over $y \in K_q$, $f(u)=\frac{u^2}{2}+O(u^3)$. Consequently,
    \begin{align*}
        f(u) - f(3u/4) &= \frac{u^2}{2} -\frac12\left(\frac{3u}{4}\right)^2 +O(u^3) = \frac{7}{32} u^2 + O(u^3).
        \end{align*}
    Hence, there exist constants $\delta \in (0,1]$ and $c_0 > 0$, depending only on $q$, such that
    $$f(u) - f(3u/4) \geq c_0 u^2$$
    for every $y \in K_q$ and every $u \in [0,\delta]$.

    To deal with larger values of $u$, we show that $g(u) := f(u) - f(3u/4)$ is increasing. Observe that $f$ is strictly convex with $f'(0) = 0$, so we have $f'(u) > 0$ for all $u > 0$. Then
    $$g'(u) = f'(u) - \frac{3}{4} f'(3u/4) \geq f'(3u/4) - \frac{3}{4} f'(3u/4) = \frac{1}{4} f'(3u/4) \geq 0.$$
    Thus $g$ is increasing on $[0,\infty)$. Set $c:=c_0 \delta^2$. If $0 \leq u \leq \delta$, then
    $$g(u)\geq c_0u^2\geq cu^2.$$
    If $\delta \leq u \leq 1$, then monotonicity gives
    $$g(u) \geq g(\delta) \geq c_0 \delta^2=c \geq cu^2.$$
    Finally, if $u\geq1$, then
    $$g(u)\geq g(\delta)\geq c
=c\min\{u^2,1\}.$$
Therefore, $g(u)\geq c\min\{u^2,1\}$ for every $u \geq 0$. Substituting $u = t \sigma$ and exponentiating completes the proof.
\end{proof}

\begin{proof}[Proof of Theorem~\ref{thm:time-dep-conv-rate}]
    First we treat the case where $p(0) \neq p^*$. The idea is to show that all the odds ratios between outcomes are bounded in expectation. In particular, if we let $\mathcal O_{i,j}(n) := \frac{p_j(n)}{p_i(n)} = \frac{p_j(0)}{p_i(0)} e^{a(n)(b_iN_i(n) - b_jN_j(n))}$ and $V_n := \sum_{i,j} \mathcal O_{i,j}(n)$, then we claim that $\sup_n \E[V_n] < \infty$.

    Given this claim, we can control the weighted differences between counts because
    \begin{align*}
        V_n &\geq \max_{i,j} \mathcal O_{i,j}(n) \\
        &\geq \frac{\min_j p_j(0)}{\max_i p_i(0)} e^{a(n)(\max_{i,j} b_iN_i(n) - b_j N_j(n))},
    \end{align*}
    which we can rearrange as
    $$\max_{i,j} b_iN_i(n) - b_j N_j(n) \leq \frac{1}{a(n)} \log \left( \frac{\max_j p_j(0)}{\min_i p_i(0)} V_n \right).$$
    This directly bounds the total variation distance, as
    \begin{align*}
        \| \mu(n) - p^* \|_{\operatorname{TV}} &= \frac{1}{2n} \sum_i | N_i(n) - np_i^*| \\
        &= \frac{1}{2n} \sum_i \frac{1}{b_i} \left| b_i N_i(n) - \frac{n}{\sum_j 1/b_j} \right|.
        \shortintertext{Notice that $n/\sum_j 1/b_j = \sum_i p_i^* (b_i N_i(n))$ is an average of the $b_i N_i$, so these distances are upper bounded by the diameter of the set of $b_i N_i$.}
        &\leq \frac{1}{2n} \left( \sum_i \frac{1}{b_i} \right) \max_{i,j} b_i N_i(n) - b_j N_j(n) \\
        &\leq \frac{1}{2na(n)} \left( \sum_i \frac{1}{b_i} \right) \left( \log \frac{\max_j p_j(0)}{\min_i p_i(0)}  + \log V_n \right).
    \end{align*}
    Taking expectations and applying Jensen's inequality, we get
    $$\mathbb E\|\mu(n) - p^*\|_{\operatorname{TV}} \leq \frac{1}{2na(n)} \left( \sum_i \frac{1}{b_i} \right) \left( \log \frac{\max_j p_j(0)}{\min_i p_i(0)}  + \log \mathbb E[V_n] \right).$$

    To prove the claim, we prove a recursive drift bound. To compare the odds ratios at times $n$ and $n+1$, we compare both to an intermediate version where we pretend that $a(n+1) = a(n)$:
    \begin{align*}
        &\E[\mathcal O_{i,j}(n+1) - \mathcal O_{i,j}(n) \mid N(n)] \\
        &\qquad = \E \left [\mathcal O_{i,j}(n+1) - \frac{p_j(0)}{p_i(0)} e^{a(n)(b_iN_i(n+1) - b_j N_j(n+1))} \mid N(n) \right] \\
        &\qquad \qquad + \E \left [\frac{p_j(0)}{p_i(0)} e^{a(n)(b_iN_i(n+1) - b_j N_j(n+1))} - \mathcal O_{i,j}(n) \mid N(n) \right].
        \intertext{We can bound the former expectation in two ways. First, when $b_i N_i(n+1) \geq b_j N_j(n+1)$, $e^{a(n+1)(b_i N_i(n+1) - b_j N_j(n+1))} \leq e^{a(n)(b_i N_i(n+1) - b_j N_j(n+1))}$, meaning we can bound those terms by 0. On the other hand, when $b_i N_i(n+1) < b_j N_j(n+1)$, we can use the mean value theorem in the form $e^{-a(n+1)c} - e^{-a(n)c} \leq \sup_{a(n+1) \leq x \leq a(n) } (a(n) - a(n+1)) ce^{-c x} \leq \frac{a(n) - a(n+1)}{e a(n+1)}$ for $c > 0$. The latter expectation can be evaluated as usual.}
        &\leq \frac{p_j(0)}{p_i(0)} \frac{a(n)-a(n+1)}{a(n+1)} + p_i(n) \mathcal O_{i,j}(n) (e^{a(n)b_i} - 1) + p_j(n) \mathcal O_{i,j}(n) (e^{-a(n)b_j} - 1) \\
        &= \frac{p_j(0)}{p_i(0)} a(n) (a(n+1)^{-1} - a(n)^{-1}) + p_j(n)(e^{a(n)b_i} - 1) + p_j(n) \mathcal O_{i,j}(n) (e^{-a(n)b_j} - 1).
    \end{align*}
    So for all sufficiently large $n$, there exist constants $C,C',C'' >0$, depending only on the $p_i(0), b_i$, and $\sup_n a(n+1)^{-1} - a(n)^{-1}$ for all $i$, such that
    $$\E[\mathcal O_{i,j}(n+1) - \mathcal O_{i,j}(n) \mid N(n)] \leq C a(n) + C' p_j(n) a(n) - C'' p_{j}(n) \mathcal O_{i,j}(n) a(n).$$
    Summing over all $i,j$, we get
    \begin{align*}
        \E[V_{n+1} - V_n \mid N(n)] &\leq C d^2 a(n) + C' d a(n) - C'' a(n) \sum_j p_j(n)^2  \sum_i \frac{1}{p_i(n)} \\
        &= C d^2 a(n) + C' d a(n) - C'' a(n) V_n \sum_j p_j(n)^2.
        \shortintertext{Using Cauchy-Schwarz, we get}
        &\leq C d^2 a(n) + C' d a(n) - \frac{C''}{d} a(n) V_n.
    \end{align*}
    Taking expectations of both sides, we get that
    $$\E[V_{n+1}] \leq (Cd^2 + C'd) a(n) + \left( 1 - \frac{C''}{d}a(n) \right) \E[V_n].$$
    Therefore, $\sup_n \E[V_{n}] < \infty$, as claimed.

    Now, we treat the case where $p(0) = p^*$. We begin similarly to before but now apply Cauchy-Schwarz:
    \begin{align*}
        \E \| \mu(n) - p^* \|_{\operatorname{TV}} &= \frac{1}{2n} \sum_i \frac{1}{b_i} \left| b_i N_i - \frac{n}{\sum_j 1/b_j} \right| \\
        &\leq \frac{1}{2n} \sqrt{\sum_i \frac{1}{b_i}} \sqrt{\sum_i \frac{1}{b_i} \left( b_i N_i - \frac{n}{\sum_j 1/b_j} \right)^2}.
        \shortintertext{Denote $Q_n := \sum_i \frac{1}{b_i} ( b_i N_i - \frac{n}{\sum_j 1/b_j})^2$. We can bound $Q_n$ by looking at an exponential of it: $a(n) Q_n/4 \leq e^{a(n) Q_n/4}$.}
        &\leq \frac{1}{n \sqrt{a(n)}} \sqrt{\sum_i \frac{1}{b_i}} e^{a(n) Q_n/8}.
    \end{align*}
    So it only remains to show that for $W_n := e^{a(n) Q_n/8}$, $\sup_n \E[W_n] < \infty$. We will do so via another drift argument.

    The first step is to show that if outcome $k$ is chosen, then
    \begin{align*}
        Q_{n+1} - Q_n &= \frac{1}{b_k} \left( \left( b_k(N_k(n)+1) - \frac{n+1}{\sum_j 1/b_j} \right)^2 - \left( b_k N_k(n) - \frac{n}{\sum_j 1/b_j} \right)^2 \right) \\
        & \qquad +\sum_{i \neq k} \frac{1}{b_i} \left( \left(b_i N_i(n) - \frac{n+1}{\sum_j 1/b_j} \right)^2 - \left( b_i N_i(n) - \frac{n}{\sum_j 1/b_j} \right)^2 \right) \\
        &= \frac{1}{b_k} \left( 2 \left( b_k N_k(n) - \frac{n}{\sum_j 1/b_j} \right) \left( b_k - \frac{1}{\sum_j 1/b_j} \right) + \left( b_k - \frac{1}{S} \right)^2 \right) \\
        & \qquad + \sum_{i \neq k} \frac{1}{b_i} \left( 2 \left( b_i N_i(n) - \frac{n}{\sum_j 1/b_j} \right) \left( - \frac{1}{\sum_j 1/b_j} \right) + \frac{1}{(\sum_j 1/b_j)^2} \right) \\
        &= -2 \sum_i \left( b_i N_i - \frac{n}{\sum_j 1/b_j} \right) p_i^* + 2 \left( b_k N_k(n) - \frac{n}{\sum_j 1/b_j} \right) + b_k - \frac{2 - p_k^*}{\sum_j 1/b_j} + \sum_{i \neq k} \frac{p_i^*}{\sum_j 1/b_j} \\
        &= 2 \left( b_k N_k(n) - \frac{n}{\sum_j 1/b_j} \right) + b_k - \frac{1}{\sum_j 1/b_j} - 2 \sum_i \left( \frac{N_i}{\sum_j 1/b_j} - \frac{n}{\sum_j 1/b_j} p_i^* \right) \\
        &= 2 \left( b_k N_k(n) - \frac{n}{\sum_j 1/b_j} \right) + b_k - \frac{1}{\sum_j 1/b_j}.
    \end{align*}
    Thus, we can upper bound the one step drift. Denoting $A_i := b_i N_i(n) - \frac{n}{\sum_j 1/b_j}$,
    \begin{align*}
        \E[W_{n+1} \mid N(n)] &\leq \E[e^{a(n) Q_{n+1}/8} \mid N(n)] \\
        &= e^{a(n) Q_n/8} \sum_i p_i(n) e^{(a(n)/8)(2A_i + b_i - 1/\sum_j 1/b_j)}.
        \shortintertext{Using $p_i(n) = \frac{p_i^* e^{-a(n) A_i}}{\sum_j p_j^* e^{-a(n) A_j}}$,}
        &\leq W_n e^{C a(n)} \frac{\sum_i p_i^* e^{-a(n) A_i} e^{(a(n)/4)A_i}}{\sum_j p_j^* e^{-a(n) A_j}}.
        \shortintertext{Apply Lemma~\ref{lem:tilting-bound} with $q_i = p_i^*$, $x_i = A_i$, and $t = a(n)$.}
        &\leq W_n e^{C a(n) - C' \min\{ a(n)^2 Q/\sum_j 1/b_j, 1\}}.
    \end{align*}
    Now we split into two cases: Let $K > 0$ be large enough such that $C'K/\sum_j 1/b_j \geq C + 1$.
    \begin{enumerate}
        \item[(i)] If $Q_n \geq K/a(n)$, then for all large enough $n$, we get
        \begin{align*}
            \E[W_{n+1} \mid N(n)] & \leq W_n e^{C a(n) - C'a(n) K/\sum_j 1/b_j} \\
            &\leq W_n e^{-a(n)} \\
            &\leq W_n \left(1- \frac{a(n)}{2} \right).
        \end{align*}

        \item[(ii)] If $Q_n < K/a(n)$, then for all large enough $n$, we get
        \begin{align*}
            \E[W_{n+1} \mid N(n)] &\leq W_n e^{C a(n)} \\
            &\leq W_n(1 + 2 C a(n)) \\
            &= \left( 1 - \frac{a(n)}{2} \right) W_n + \left( \frac{a(n)}{2} + 2 C a(n) \right) W_n \\
            &\leq \left( 1 - \frac{a(n)}{2} \right) W_n + \left( \frac{1}{2} + 2 C \right) e^{K/8} a(n).
        \end{align*}
    \end{enumerate}
    So we get an overall bound of
    $$\E[W_{n+1} \mid N(n)] \leq \left( 1 - \frac{a(n)}{2} \right) W_n + \left( \frac{1}{2} + 2 C \right) e^{K/8} a(n)$$
    for all sufficiently large $n$, and taking expectations gives
    $$\E[W_{n+1}] \leq \left( 1 - \frac{a(n)}{2} \right) \E[W_n] + \left( \frac{1}{2} + 2 C \right) e^{K/8} a(n)$$
    for all sufficiently large $n$. As before, we get $\sup_n \E[W_n] < \infty$.
\end{proof}

\proofappendixheading{Proof of Theorem~\ref{thm:limit_probabilities} (Pathwise Convergence Rate)}
\proofappendixstatement{Theorem}{For the self-balancing sampler,
\[
\|\mu(n)-p^*\|_{\operatorname{TV}}\xrightarrow{a.s.}0,
\qquad
\|\mu(n)-p^*\|_{\operatorname{TV}}\lesssim \frac{\log n}{n} \sum_{i=1}^d \frac{1}{\beta_i}
\quad\text{a.s.}
\]}
\begin{proof}
    The idea is to show that the log odds ratios of the sampling probabilities never get too large; this will force $\beta_i N_i \approx \beta_j N_j$, which will make $N_i/n \propto 1/\beta_i$. In particular, let
    $$R_{i,j} \coloneqq -\log \frac{p_i(n)}{p_j(n)} + \log \frac{p_i(0)}{p_j(0)} = \beta_i N_i - \beta_j N_j.$$
    It is very unlikely for $R_{i,j}$ to exceed $2\log n$, as the following one-step analysis shows:
    \begin{align*}
        \mathbb P(X_{n+1} = i, R_{i,j}(n) \geq 2 \log n) &= \E[ \mathbb P(X_{n+1} = i, R_{i,j}(n) \geq 2 \log n \mid p(n))] \\
        &= \mathbb E[p_i(n) \mathbf 1_{\{R_{i,j}(n) \geq 2 \log n\}}] \\
        &\leq \E \left[p_j(n) \frac{p_j(0)}{p_i(0)} \frac{1}{n^2} \mathbf 1_{\{R_{i,j}(n) \geq 2 \log n\}} \right] \\
        &\lesssim \frac{1}{n^2}.
    \end{align*}
    Also note that $R_{i,j}$ can only jump by at most $\max_k \beta_k$ at a time. Therefore, 
    \[\P(R_{i,j}(n) \geq 2\log(n) + 2\max_k \beta_k \quad \textrm{i.o.}) \leq \P(X_{n+1} = i, R_{i,j}(n) \geq 2\log(n) \quad \textrm{i.o.}) = 0,\]
    \noindent by the Borel-Cantelli theorem. Applying this argument to $R_{j,i} = -R_{i,j}$, we get $|R_{i,j}(n)| \leq 2 \log n + 2\max_k \beta_k$ all but finitely often with probability one.

    Now analyze the counts in terms of the $R_{i,1}$:
    $$\frac{N_i}{n} = \frac{1}{\beta_i} \cdot \frac{\beta_1 N_1}{n} + \frac{R_{i,1}}{\beta_i n},$$
    and summing over $i$ gives
    $$1 = \sum_{k=1}^d \frac{N_k}{n} = \frac{\beta_1 N_1}{n} \left( \sum_{k=1}^d \frac{1}{\beta_k} \right) + \frac{1}{n} \sum_{k=1}^d \frac{R_{k,1}}{\beta_k},$$
    which we rearrange into
    $$\frac{\beta_1 N_1}{n} = \frac{1}{\sum_{k=1}^d 1/\beta_k} \left(1 - \frac{1}{n} \sum_{k=1}^d \frac{R_{k,1}}{\beta_k} \right).$$
    Plugging this back into the expression for $N_i/n$ gives
    $$\frac{N_i}{n} = \underbrace{\frac{1/\beta_i}{\sum_{k=1}^d 1/\beta_k}}_{p_i^*} \left(1 - \frac{1}{n} \sum_{k=1}^d \frac{R_{k,1}}{\beta_k} \right) + \frac{R_{i,1}}{\beta_i n}.$$
    So the total variation distance is
    \begin{align*}
        \| \mu(n) - p^* \|_{\operatorname{TV}} &= \frac{1}{2} \sum_{i=1}^d \left| \frac{N_i}{n} - p_i^* \right| \\
        &= \frac{1}{2n} \sum_{i=1}^d \left| \frac{R_{i,1}}{\beta_i} - p_i^*\sum_{k=1}^d \frac{R_{k,1}}{\beta_k} \right| \\
        &\lesssim \frac{\log n}{n} \sum_{i=1}^d \frac{1}{\beta_i}.\qedhere
    \end{align*}
\end{proof}

\newpage
\subsection{Section 2.4 Proofs: Optimization Formulation}
\label{appendix:2_4_optimization_proofs}

\proofappendixheading{Proof of Theorem~\ref{thm:general_optimization_problem} (Entropy-Regularized + Squared-Error Optimization)}
\proofappendixstatement{Theorem}{Let $L_\beta(n)=\frac12\sum_{i=1}^d \beta_i(N_i(n)-np_i^*)^2$ and let $D_\beta(q)=\E[L_\beta(n+1)-L_\beta(n)\mid N(n),q]$. Then the optimization problem
\[
\argmax_{q\in\Delta^d}\bigl(H(q)-H(q,p(0))-D_\beta(q)\bigr)
=
\argmin_{q\in\Delta^d}\bigl(\operatorname{KL}(q\|p(0))+D_\beta(q)\bigr)
\]
has unique solution $q_i\propto p_i(0)e^{-\beta_i(N_i(n)+1/2)}$.}

\begin{proof}
    First, expand $D$ to isolate the terms depending on $q$:
    \begin{align*}
        D(q) &= \frac{1}{2}\sum_{i=1}^d \beta_i\E [N_i(n+1)^2 - 2(n+1) p_i^* N_i(n+1) + ((n+1) p_i^*)^2 \\
        &\qquad \qquad \qquad - N_i(n)^2 + 2n p_i^* N_i(n) - (np_i^*)^2  \mid N(n),q ] \\
        &= \frac{1}{2}\sum_{i=1}^d \beta_i[q_i (N_i(n) + 1)^2 + (1-q_i) N_i(n)^2 - N_i(n)^2 \\
        &\qquad \qquad \qquad- 2 p_i^*((n+1)(N_i(n) + q_i) - n N_i(n)) + (2n+1)(p_i^*)^2] \\
        &= \frac{1}{2}\sum_{i=1}^d \beta_i[2q_i N_i(n) + q_i - 2(n+1) p_i^* q_i - 2p_i^* N_i(n) + (2n+1)(p_i^*)^2] \\
        &= \sum_{i=1}^d \left( \beta_i N_i(n) q_i + \frac{1}{2} \beta_iq_i + \frac{2n+1}{2} \beta_i (p_i^*)^2 - \frac{N_i(n)}{\sum_{k=1}^d 1/\beta_k} \right) - \frac{n+1}{\sum_{k=1}^d 1/\beta_k},
    \end{align*}
    so it is equivalent to maximize $H(q) - \sum_{i=1}^d \beta_i N_i(n) q_i + \frac{1}{2} \beta_iq_i$. When all the $\beta_i$ are equal, this just looks like maximizing $H(q) - \beta \langle N(n),q \rangle$.

    Now solve the optimization problem using the Lagrange multiplier $\lambda$. The Lagrangian is
    $$\mathcal L(q) \coloneqq \sum_{i=1}^d \left(- q_i \log q_i + q_i \log p_i(0) - \beta_i N_i q_i - \frac{1}{2} \beta_i q_i + \lambda q_i \right) - \lambda.$$
    Setting $\nabla_q \mathcal L(q) = 0$ gives
    $$0 = -\log q_i - 1 + \log p_i(0) - \beta_i N_i - \frac{\beta_i}{2} + \lambda \qquad \implies \qquad q_i = p_i(0)e^{-\beta_i N_i - \beta_i/2 + \lambda - 1}.$$
    Therefore, $q_i \propto p_i(0)e^{-\beta_i (N_i(n)+1/2)}$, as claimed.
\end{proof}

\proofappendixheading{Proof of Theorem~\ref{thm:instantaneous_optimization} (Instantaneous Optimization)}
\proofappendixstatement{Theorem}{If $L_\beta(x)=\frac12\sum_{i=1}^d \beta_i(x_i-np_i^*)^2$, then
\[
\argmin_{q\in\Delta^d}\Bigl(\langle \nabla L_\beta(N(n)),q\rangle+\operatorname{KL}(q\|p(0))\Bigr)
\]
has unique solution $q_i\propto p_i(0)e^{-\beta_iN_i(n)}$.}
\begin{proof}
    Calculate the gradient:
    \begin{align*}
        \frac{\partial L_\beta}{\partial x_i}(x) &= \beta_i(x_i - n p_i^*).   
    \end{align*}
    so we aim to optimize
    $$\sum_{i=1}^d \beta_i(N_i(n)-np_i^*)q_i + \sum_{i=1}^d q_i \log q_i - q_i \log p_i(0).$$
    Define the Lagrangian
    $$\mathcal L(q) \coloneqq \left(\sum_{i=1}^d \beta_i(N_i(n)-np_i^*)q_i + q_i \log q_i - q_i \log p_i(0) + \lambda q_i\right) - \lambda$$
    with Lagrange multiplier $\lambda$. Setting $\nabla_q \mathcal L(q) = 0$ gives
    \begin{align*}
    0 &= \beta_i(N_i(n)-np_i^*) + \log q_i + 1 - \log p_i(0) + \lambda, \\
    &= \beta_i N_i(n) - \frac{n}{\sum_{j=1}^d 1/\beta_j} + \log q_i + 1 - \log p_i(0) + \lambda,
    \end{align*}
    which we solve to get
    \begin{equation*}
        q_i \propto p_i(0) e^{-\beta_i N_i(n)}.\qedhere
    \end{equation*}
\end{proof}

\proofappendixheading{Proof of Theorem~\ref{thm:stochastic_mirror_descent} (Stochastic Mirror Descent)}
\proofappendixstatement{Theorem}{The optimization problem
\[
\argmin_{q\in\Delta^d}\left(h\left\langle \beta_{X_n}e_{X_n}-\frac{1}{\sum_{j=1}^d\beta_j^{-1}}\mathbf 1,\;q\right\rangle+\operatorname{KL}(q\|p(n-1))\right)
\]
has unique solution $q=p(n)$, i.e. $q_i\propto p_i(0)e^{-h\beta_iN_i(n)}$.}
\begin{proof}
    Since $\langle \mathbf 1,q \rangle = 1$, it is equivalent to minimize
    $$h \langle \beta_{X_n} e_{X_n},q \rangle + \operatorname{KL}(q \| p(n-1)) = h \beta_{X_n} q_{X_n} + \sum_{i=1}^d q_i \log q_i - q_i \log p_i(n-1).$$
    Defining the Lagrangian
    $$\mathcal L(q) \coloneqq h \beta_{X_n} q_{X_n} - \lambda + \sum_{i=1}^d q_i \log q_i - q_i \log p_i(n-1) + \lambda q_i$$
    with Lagrange multiplier $\lambda$, the condition $\nabla_q \mathcal L(q) = 0$ gives
    $$0 = \begin{cases}
        h \beta_i + \log q_i + 1 - \log p_i(n-1) + \lambda &\text{if $X_n=i$}, \\
        \log q_i + 1 - \log p_i(n-1) + \lambda &\text{if $X_n \neq i$}.
    \end{cases}$$
    Solving for $q_i$ gives
    $$q_i \propto \begin{cases}
        p_i(n-1) e^{-h\beta_i} &\text{if $X_n = i$}, \\
        p_i(n-1) &\text{if $X_n \neq i$},
    \end{cases}$$
    which is precisely the update rule for the self-balancing sampler.
\end{proof}

\newpage
\subsection{Section 2.5 Proofs: Expected Cover Time}
\label{appendix:2_5_expected_cover_time}

\proofappendixheading{Proof of Theorem~\ref{thm:expected_cover_time} (Expected Cover Time)}
\proofappendixstatement{Theorem}{If $C=\max_{1\le i\le d}\min\{n\ge 1:X_n=i\}$ and $p_i(0)=1/d$, then
\[
\E[C]=\sum_{i=1}^d \frac{\log(1+\beta_i\log d)}{\beta_i}+O(d).
\]}
\begin{proof}
    Writing $C = C_1 + \cdots + C_d$, where $C_i \coloneqq |\{ 1 \leq n \leq C : X_n = i\}|$ is the number of times we see outcome $i$ before covering all the outcomes, we need only show 
    $$\E[C_i] = \frac{\log(1 + \beta_i \log d)}{\beta_i} + O(1).$$
    \noindent The key insight is to embed the self-balancing sampler into a continuous time stochastic process with the sampling order determined by independent exponential clocks. For $1 \leq i \leq d$ and $j \geq 0$, let $W_{i,j} \sim \operatorname{Exp}(\alpha_i^j)$ be independent, and let $Y_{i,j} \coloneqq \sum_{k=0}^j W_{i,k}$. Then the order of the $Y_{i,j}$ and their associated $i$ indices has the same distribution as the order of the $X_n$, e.g.\ $\P(Y_{1,0} < Y_{1,1} < Y_{6,0} < \cdots) = \P(X_1 = 1, X_2 = 1, X_3 = 6, \dots)$. Letting 
    
    $$\widetilde C_i(t) \coloneqq \min \{ j \geq 0 : Y_{i,j} > t \},$$
    
    \noindent be the number of times index $i$ is sampled by time $t$, we claim
    $$\frac{\log(1 + \beta_i t)}{\beta_i} \leq \E[\widetilde C_i(t)] \leq \frac{\log(1 + \beta_i t)}{\beta_i} + 1.$$

    The lower bound is equivalent to $1 + \beta_i t \leq e^{\beta_i \E[\widetilde C_i(t)]}$, which follows from
    \begin{align*}
        e^{\beta_i \E[\widetilde C_i(t)]} &= 1 + \int_0^t \frac{d}{du}e^{\beta_i \E[\widetilde C_i(u)]} \Bigg|_{u=s} \, ds \\
        &= 1 + \beta_i \int_0^t \frac{d}{du}\E[\widetilde C_i(u)] \Bigg|_{u=s} \cdot e^{\beta_i \E[\widetilde C_i(s)]} \, ds \\
        &= 1 + \beta_i \int_0^t \E[e^{-\beta_i \widetilde C_i(s)}] e^{\beta_i \E[\widetilde C_i(s)]} \, ds \\
        &\geq 1 + \beta_i \int_0^t e^{-\beta_i \E[\widetilde C_i(s)]} e^{\beta_i \E[\widetilde C_i(s)]} \, ds \\
        &= 1 + \beta_i t,
    \end{align*}
    where we used that the derivative of the expectation of a birth process is the expected birth rate. For the upper bound, we have
    \begin{align*}
        e^{\beta_i \E[\widetilde C_i(t)]} &\leq \E[e^{\beta_i \widetilde C_i(t)}] \\
        &= 1 + \int_0^t \frac{d}{du} \E[e^{\beta_i \widetilde C_i(u)}] \Bigg|_{u=s} \, ds.
        \shortintertext{By Kolmogorov's forward equation, $\frac{d}{du} \E[e^{\beta_i \widetilde C_i(u)}]$ is the expectation of the generator of the birth process $\widetilde C_i(t)$ applied to the function $f(z) = e^{\beta_i z}$.}
        &= 1 + \int_0^t \E[e^{-\beta_i \widetilde C_i(s)}(e^{\beta_i (\widetilde C_i(s)+1)} - e^{\beta_i \widetilde C_i(s)})] \, ds \\
        &= 1 + \int_0^t e^{\beta_i} - 1 \, ds \\
        &= 1 + (e^{\beta_i} - 1)t.
    \end{align*}
    This gives an upper bound of $\E[\widetilde C_i(t)] \leq \frac{\log(1 + (e^{\beta_i} - 1) t)}{\beta_i} \leq \frac{\log(1 + \beta_i t)}{\beta_i} + 1$.

    Now, we incorporate the random time in the continuous process which corresponds to the cover time in the self-balancing sampler. Let $M_{-i} \coloneqq \max \{ W_{j,0} : j \neq i \}$ be the maximum
    waiting time for all $j \neq i$ to be sampled at least once, and note
    $$C_i \overset{d}{=} \max \{ 1, \widetilde C_i(M_{-i}) \},$$ 
    \noindent so that
    $$\E[C_i] = \E[\widetilde C_i(M_{-i})] + \underbrace{\P(W_{i,0} > W_{j,0} \: \forall j \neq i)}_{=1/d}.$$
    Applying our previous inequalities gives
    $$\left|\E[C_i] - \E \left[\frac{\log(1 + \beta_i M_{-i})}{\beta_i}  \right] \right| \leq 1 + \frac{1}{d},$$
    and all that remains now is to show that $\E[\widetilde C_i(M_{-i})] \approx \E[\widetilde C_i(\log d)]$, where $\E[M_{-i}] = \sum_{k=1}^{d-1} \frac{1}{k} \approx \log d$. We have
    \begin{align*}
        \left|\E[C_i] - \frac{\log(1 + \beta_i \log d)}{\beta_i} \right| &\leq \left|\E[C_i] - \E \left[\frac{\log(1 + \beta_i M_{-i})}{\beta_i}  \right] \right| \\
        &\qquad \qquad + \left|\E \left[\frac{\log(1 + \beta_i M_{-i})}{\beta_i}  \right]  - \frac{\log(1 + \beta_i \log d)}{\beta_i} \right|.
        \intertext{The latter term can be bounded using the fact that $g(x) \coloneqq \frac{\log(1 + \beta_i x)}{\beta_i}$ is $1$-Lipschitz, which follows from $|g'(x)| = |\frac{1}{1+\beta x}| \leq 1$ for all $x \geq 0$.}
        &\leq 1 + \frac{1}{d} + \E[|M_{-i} - \log d|] \\
        &\leq 1 + \frac{1}{d} + \sqrt{\E[(M_{-i} - \log d)^2]} \\
        &= 1 + \frac{1}{d} + \sqrt{(\E[M_{-i} - \log d])^2 + \Var(M_{-i})} \\
        &= 1 + \frac{1}{d} + \sqrt{\left(\log d - \sum_{k=1}^{d-1} \frac{1}{k}\right)^2 + \sum_{k=1}^{d-1} \frac{1}{k^2}} \\
        &\leq 1 + \frac{1}{d} + \sqrt{1 + \frac{\pi^2}{6}}.\qedhere
    \end{align*}
\end{proof}

\newpage
\subsection{Section 3.1 Proofs: Diffusion Approximations}
\label{appendix:3_1_diffusion_limit_proofs}

\proofappendixheading{Proof of Proposition~\ref{prop:projection_centered_diffusion_limit} (Projection-Centered Diffusion Limit)}
\proofappendixstatement{Proposition}{Fix $i\in[d]$. If $Z_0^h\to z_0$, then $Z^h\Rightarrow Z$ in $D([0,T],\R)$ for each $T<\infty$, where $Z$ solves
\[
dZ_t=-\frac{\beta_i\beta_d}{\beta_i+\beta_d}Z_t\,dt+\sqrt{\beta_i\beta_d}\,dW_t,
\qquad Z_0=z_0.
\]}

\begin{proof}
For convenience, drop the subscript on $c_\beta$ and let $q_i := p_i(0)$. Note:
\[\frac{p_i^{\beta h}}{p_i^{\beta h} + p_d^{\beta h}} = \frac{q_ie^{-\beta_i h N_i}}{q_ie^{-\beta_i h N_i} + q_de^{-\beta_d h N_d}} = \frac{e^{-Y_i^{\beta h}}}{q_d/q_i + e^{-Y_i^{\beta h}}}\]
\noindent Since $Z_{nh}^h = z$ implies $Y^{h\beta}_i(n) = z \sqrt{h} + c_\beta$, we have increments:
\[\Delta Z_{nh}^h :=  Z_{(n+1)h}^h - Z_{nh}^h \mid \{Z_{nh}^h = z\} = \begin{cases}
    \beta_i\sqrt{h} & \text{w.p. } \frac{e^{-( z\sqrt{h} + c) }}{q_d/q_i + e^{-( z\sqrt{h} + c)}}, \\[5pt]
    -\beta_d\sqrt{h} & \text{w.p. } \frac{q_d/q_i}{q_d/q_i + e^{-( z\sqrt{h} + c)}}, \\ 
\end{cases}\]
\noindent so since $e^{-c} = \frac{\beta_dq_d}{\beta_iq_i}$, our process has drift:
\begin{align*}
    \mu_h(z) & := \E [ \Delta Z_{nh}^h \mid Z_{nh}^h = z]\\
    & = \frac{\sqrt{h}}{q_d/q_i + e^{- (z \sqrt{h} + c)}}\left[\beta_i e^{- (z \sqrt{h} + c)} - \beta_d q_d/q_i\right] \\
    & = \frac{\beta_d\sqrt{h}}{1 + \frac{\beta_d}{\beta_i} e^{- z \sqrt{h}}} \left[e^{- z \sqrt{h}} - 1\right].
\end{align*}
\noindent Hence:
\[\mu_h(z) = -\frac{\beta_d }{1 + \frac{\beta_d}{\beta_i}} zh + O(z^2h^{3/2}) = -\frac{z}{\beta_d^{-1} + \beta_i^{-1}} h + O(z^2h^{3/2}).\]
\noindent Hence $h^{-1} \mu_h(z)$ converges uniformly for $z$ in a compact set. Similarly, for the variance
\begin{align*}
    \sigma^2_h(z) & := \E [ (\Delta Z_{nh}^h)^2 \mid Z_{nh}^h = z] - \mu_h(z)^2 \\[5pt]
    & = \frac{h}{q_d/q_i + \frac{\beta_dq_d}{\beta_iq_i}e^{-z \sqrt{h}}} \left[\beta_i^2 \cdot \frac{\beta_dq_d}{\beta_iq_i} e^{- z \sqrt{h}} + \beta_d^2 \cdot \frac{q_d}{q_i}\right] + O(z^2h^2)\\[5pt]
    & = \frac{\beta_i\beta_d + \beta_d^2}{1 + \frac{\beta_d}{\beta_i}}h + O(z^2h^{3/2})\\[5pt]
    & = \beta_i \beta_d h +O(z^2h^{3/2}).
\end{align*}

\noindent Again $h^{-1} \sigma_h(z)$ converges uniformly for $z$ in a compact set. Lastly, since
    \[|\Delta Z_{nh}^h|~\leq \max(\beta_i, \beta_d) \sqrt{h} \to 0,\]
\noindent we have no large jumps. Namely, for every $\epsilon > 0$
    \[\P(|\Delta Z_{nh}^h| > \epsilon \mid Z_{nh}^h = z) = o(h), \]
independently of $z$, as this probability is zero for sufficiently small $h$. As the diffusion
\[dZ_t = -\frac{\beta_i \beta_d}{\beta_i + \beta_d} Z_t dt + \sqrt{\beta_i\beta_d} \cdot dW_t, \quad Z_0 = z_0,\]
has globally Lipschitz coefficients, there exists a unique strong solution, namely just the OU process. Standard theory shows the desired convergence. See for example Section 8.7 of \cite{durrett1996stochastic}. 
\end{proof}

\proofappendixheading{Proof of Proposition~\ref{prop:mean_centered_diffusion_limit} (Mean-Centered Diffusion Limit)}
\proofappendixstatement{Proposition}{If $Z_0^h\to z_0$, then $Z^h\Rightarrow Z$ as $h\downarrow 0$, where
\[
dZ_t=-\frac{1}{S_\beta}Z_t\,dt+\Sigma^{1/2}dW_t,
\quad
S_\beta=\sum_j\beta_j^{-1}, \quad D_\beta = \operatorname{diag}(\beta)
\quad
\Sigma=S_\beta^{-1}D_\beta-S_\beta^{-2}\mathbf 1\mathbf 1^T,
\]
and the limiting diffusion is supported on $\{z:\langle z,p^*\rangle=0\}$.}

\begin{proof}
For ease, drop the subscript of $c_\beta$ and let $q_i := p_i(0)$. Since $\beta_i p_i^*$ is independent of $i$
\[p^{h\beta}_i = \frac{q_ie^{-h\beta_i N_i}}{\sum_j q_je^{-h\beta_j N_j}} = \frac{q_ie^{-M_i^{h \beta}}}{\sum_j q_je^{-M_j^{h\beta}}},\]
\noindent so the transitions are a function of our mean-centered counts alone. The value $c_\beta$ is chosen so that $e^{- c_i} = \kappa (q_i\beta_i)^{-1}$ for $\kappa > 0$ independent of $i$, and so $\langle c_\beta, p^* \rangle = 0$ satisfies the same constraint as $M^\beta$. Since $Z_{nh}^h = z$ implies $M^{\beta h}_n = z\sqrt{h} + c$, note for standard basis vector $e_i$
    \[\P\left(|\Delta Z_{nh}^h| =  \sqrt{h} D_\beta(e_i - p^*) \right) = \frac{q_ie^{- (z_i \sqrt{h} + c_i)} }{\sum_{j=1}^d q_je^{- (z_j \sqrt{h} + c_j)}} = \frac{e^{- z_i \sqrt{h}} }{\sum_{j=1}^d \frac{1}{\beta_j}e^{- z_j \sqrt{h} }}.\]
Hence this process has drift
\begin{align*}
    (\mu_h(z))_i & := \E [(\Delta Z_{nh}^h)_i \mid Z^h_{nh} = z] \\[5pt]
    & = \sqrt{h} \left[\frac{e^{- z_i \sqrt{h}} }{\sum_{j=1}^d \frac{1}{\beta_j}e^{- z_j \sqrt{h} }} - \beta_ip_i^*\right]\\[5pt]
    & = \sqrt{h} \left[\frac{1- z_i \sqrt{h}}{\sum_{j=1}^d \frac{1}{\beta_j}(1 - z_j \sqrt{h})} - \beta_ip_i^* + O(h)\right]\\[5pt]
    & = \sqrt{h} \left[\frac{\left(\sum_{j=1}^d p_j^*z_j-z_i\right)\sqrt{h}}{\sum_{j=1}^d \frac{1}{\beta_j}(1- z_j \sqrt{h})}  + O(h)\right]\\[5pt]
    & = \frac{1}{\sum_{j=1}^d \frac{1}{\beta_j}} \left(\sum_{j=1}^d p_j^*z_j-z_i\right) h + o(h)
\end{align*}
\noindent As above, it is easy to see the convergence is uniform for $z$ in compact sets. Similarly
\begin{align*}
     \E [(\Delta Z_{nh}^h)_i^2 \mid Z_{nh}^h = z ] & = h \beta_i^2 \left[(p_i^*)^2 + \left(1-2p_i^* \right) \cdot \frac{\frac{1}{\beta_i}e^{- z_i \sqrt{h}}}{\sum_{j=1}^d \frac{1}{\beta_j}e^{- z_j \sqrt{h}}}\right]\\[5pt]
     & = h\beta_i^2 \left[(p_i^*)^2 + \left(1-2p_i^* \right) \cdot \left(p_i^* + O(\sqrt{h})\right)\right]\\[5pt]
     & = h \beta_i^2  p_i^*(1- p_i^*) + o(h)
\end{align*}

\noindent Moreover, for $i \neq j$:
\begin{align*}
     \E [(\Delta Z_{nh}^h)_i (\Delta Z_{nh}^h)_j \mid Z_{nh}^h = z  ] & = h \beta_i \beta_j\left[p_i^*p_j^* -  \left(\frac{\frac{p_i^*}{\beta_j}e^{- z_j \sqrt{h}} + \frac{p_j^*}{\beta_i}e^{-  z_i \sqrt{h}}}{\sum_{\ell = 1}^d \frac{1}{ \beta_\ell}e^{- z_\ell\sqrt{h}}}\right)\right] \\[5pt]
     & = h \beta_i \beta_j\left[p_i^*p_j^* - \left(\frac{\frac{p_i^*}{\beta_j} + \frac{p_j^*}{\beta_i}-O(\sqrt{h})}{\sum_{\ell=1}^d \frac{1}{\beta_\ell}-O(\sqrt{h})}\right)\right]\\[5pt]
     & = - \beta_i \beta_j p_i^*p_j^*h + o(h).
\end{align*}
\noindent Thus we have the desired covariance. As the coefficients of the limiting diffusion are globally Lipschitz, the associated SDE has a unique strong solution. Again standard theory, such as Section 8.7 of \cite{durrett1996stochastic}, gives the desired convergence result. This gives a limiting diffusion
\[dZ_t = -\frac{1}{S_\beta}(Z_t - \bar{Z_t}) dt + \Sigma^{1/2} dW_t, \qquad \bar{Z}_t := \langle p^*, Z_t \rangle.\]
\noindent Lastly, we show that the limiting diffusion lives in the hyperplane
\[Z_t \in \{z : \langle z, p^* \rangle = 0\}.\]
\noindent Note that by construction
\[\langle Y^\beta(n), p^*\rangle = \langle c_\beta, p^* \rangle = 0.\]
\noindent and hence $\langle Z^h_t, p^*\rangle = 0$ for all $h$ and $t$, from which the result follows. In particular, $\bar{Z}_t = 0$, so our above diffusion simplifies to the desired form.
\end{proof}

\newpage 
\subsection{Section~\ref{sec:adversarial_interpretation} Proofs: Adversarial Error-Rate}
\label{appendix:3_3_adversarial_error_rate_proof}

\proofappendixheading{Proof of Lemma~\ref{lem:2d_stationary_project_centered} ($d=2$ Stationary Distribution)}
\proofappendixstatement{Lemma}{When $d=2$ and $\beta_1=\beta_2=\beta>0$, the projection-centered chain $\tilde Y(n)$ is reversible on $\Z$ with stationary distribution
\[
\pi(z)\propto \frac{p_2(0)+p_1(0)e^{-\beta z}}{p_1(0)+p_2(0)}
\left(\frac{p_1(0)}{p_2(0)}\right)^z
e^{-\beta {z \choose 2}}.
\]}
\begin{proof}
    In the $d=2$ case then $\tilde{Y}_n$ is a simple birth-death chain with birth probability:
\[\lambda_n = \frac{p_1(0)e^{-\beta n}}{p_2(0) + p_1(0) e^{-\beta n}}\]
    \noindent and death otherwise. The detailed balance equations then imply:
    \[\frac{\pi(n)}{\pi(n-1)} = \frac{\lambda_{n-1}}{1-\lambda_{n}}.\]
    \noindent Hence:
    \[\pi(n) = \pi_0 \prod_{j=1}^n \frac{\lambda_{j-1}}{1-\lambda_j} = \pi_0 \frac{p_2(0) + p_1(0)e^{-\beta n}}{p_1(0) + p_2(0)} \left(\frac{p_1(0)}{p_2(0)}\right)^n e^{-\beta {n \choose 2}}\]
    \noindent If $\beta > 0$ then this is normalizable as it has discrete Gaussian tails.
\end{proof}

\proofappendixheading{Proof of Lemma~\ref{lem:small_beta_adversarial} ($d=2$ Adversarial Error-Rate)}
\proofappendixstatement{Lemma}{For $d=2$, and uniform $\beta$, as $\beta \to 0$:
\[
\limsup_{n\to\infty}\frac1n\sum_{k=0}^{n-1}\P\!\left(X_k\neq \arg\min_{i=1,2}N_i(k-1)\right)
=\frac12-\Theta(\sqrt{\beta}).
\]}
\begin{proof}
    Note in the $d=2$ case with uniform $\beta$, we have:
    \[\limsup_{n \to \infty} \frac{1}{n}\sum_{k=0}^{n-1}\P\!\left(X_k \neq \argmax_{i=1,2} \pi_i(N(k-1))\right) = \limsup_{n \to \infty} \frac{1}{n}\sum_{k=0}^{n-1} \E_{y \sim \tilde{Y}(k)} \left[\frac{1}{1+e^{\beta |y|}}\right].\]
    \noindent Since $Y(n)$ is ergodic, the ergodic theorem and Lemma \ref{lem:2d_stationary_project_centered} imply:
    \[\E_{y \sim \pi}\left[ \frac{1}{1+e^{\beta |y|}} \right]  = Z^{-1}\sum_{y \in \Z} \frac{\cosh\left(\frac{\beta y}{2}\right)}{1 + e^{\beta |y|}} \cdot e^{-\frac{\beta y^2}{2}} = \frac{1}{2Z}\sum_{y \in \Z}  e^{-\frac{\beta (y^2 + |y|)}{2} }.\]
    \noindent for $Z = \sum_{z \in \Z} \cosh(\beta z/2) \cdot e^{-\beta z^2/2}$. While no simple closed-form exists, the simple bound:
    \[\frac{1}{2} -\frac{x}{4}  \leq (1 + e^x)^{-1} \leq \frac{1}{2} -\frac{x}{4} + \frac{x^3}{48}, \quad \forall x \geq 0,\]
    \noindent implies for small $\beta$ our error-rate is of order:
    \[\E_{y \sim \pi}\left[ \frac{1}{1+e^{\beta |y|}} \right] = \frac{1}{2} - \frac{\beta}{4} \E_\pi |Y| + O\left(\beta^3 E_\pi |Y|^3\right).\]
    \noindent To control this, note the stationary distribution of $Y$ is a mixture of two (discrete) Gaussians
    \[\cosh\left(\frac{\beta y}{2} \right) e^{-\beta y^2/2} \propto \frac{1}{2} \left[e^{-\frac{\beta}{2}(y-1/2)^2} + e^{-\frac{\beta}{2}(y+1/2)^2} \right].\]
    \noindent centered at $\pm \frac{1}{2}$ with variance $\beta^{-1}$. For small $\beta$, the standard Riemann sum comparison implies the distribution of $Y$ is well-approximated by the continuous counterpart. Hence $Y \approx \frac{\xi}{2} + \beta^{-1/2}G$ for a standard discrete gaussian $G$ and Rademacher $\xi \sim \mathrm{Unif}\{\pm 1\}$.
    
    We can then estimate $\E |Y|$ and $\E |Y|^3$ using the standard moments of the folded normal. In particular, as $\beta \downarrow 0$ the gaussian portion $G$ dominates and we get the moments of the half-normal distribution:
    \[\E |Y| \asymp \sqrt{\frac{2}{\pi \beta}}, \qquad \E |Y|^3 = O(\beta^{-3/2}).\]
    \noindent Combining this with the above yields:
    \[\E_{y \sim \pi}\left[ \frac{1}{1+e^{\beta |y|}} \right] = \frac{1}{2} - c \sqrt{\beta} + O(\beta^{3/2}) = \frac{1}{2} - \Theta(\sqrt{\beta}).\]
    \noindent as desired.
 \end{proof}

\proofappendixheading{Proof of Proposition~\ref{prop:average_case_optimal_control}}
\proofappendixstatement{Proposition}{Suppose $\pi_i(x)\propto a_i e^{-H_i(x)}$, that $\pi(x)=\pi(x')$ whenever $(I-p^*1^T)x=(I-p^*1^T)x'$, and that there exist $c,C>0$ such that for all $y$ with $\sum_i y_i=0$ and $y_j\le y_i$,
\[
H_i(y)-H_j(y)\ge c(y_i-y_j)-C.
\]
Then the empirical distribution associated to $\pi$ satisfies $\E\|\mu_n-p^*\|_{\operatorname{TV}}=O(n^{-1})$.}

\begin{proof}
    Let \(Y(n):=N(n)-np^*\) and \(W_n:=\|Y(n)\|_1\). By the invariance assumption,
        \[\pi(N(n))=\pi(Y(n)),\]
    so the transition probabilities may be viewed as functions of the centered counts. Define
    \[h_i(u):=|u-p_i^*|-|u|,\qquad g_i(u):=|u+1-p_i^*|-|u-p_i^*|.\]
    Then
        \[\E[W_{n+1}-W_n\mid Y(n)=y]=\sum_i h_i(y_i)+\sum_i\pi_i(y)g_i(y_i).\]
    Let \(p_{\min}^*:=\min_i p_i^*\) and \(p_{\max}^*:=\max_i p_i^*\). Since
    \[\sum_{i:y_i>0}y_i=\frac{\|y\|_1}{2},\]
    whenever \(\|y\|_1\ge 2dp_{\max}^*\), some \(y_j\ge p_j^*\), and hence
        \[\sum_i h_i(y_i)\le 1-2p_{\min}^*.\]
    Now let \(m:=\min_i y_i\) and \(S:=\{i:y_i>-1\}\). Since \(\sum_i y_i=0\),
        \[m\le -\frac{\|y\|_1}{2d}.\]
    Choosing $j$ such that $y_j=m$, for $i \in S$ the assumption on $H$ gives
    \[\pi_i(y)\le \frac{a_{\max}}{a_{\min}}e^{-c(y_i-m)+C}\le\frac{a_{\max}}{a_{\min}} e^{-c\|y\|_1/(2d)+C+c}.
    \]
    Therefore
    \[\sum_{i\in S}\pi_i(y)\le d\frac{a_{\max}}{a_{\min}}e^{-c\|y\|_1/(2d)+C+c} =:L(\|y\|_1).\]
    Since $g_i(y_i)=-1$ for $i\notin S$ and $g_i(y_i)\le 1$ otherwise,
    \[\sum_i\pi_i(y)g_i(y_i)\le 2L(\|y\|_1)-1.\]
    Thus, outside a sufficiently large compact set,
    \[
        \E[W_{n+1}-W_n\mid Y(n)=y]\le 2L(\|y\|_1)-2p_{\min}^*\le -p_{\min}^*.
    \]
    Finally,
    \[
        |W_{n+1}-W_n|
        \le
        \|e_{X_{n+1}}-p^*\|_1
        \le 2.
    \]
    Lemma~\ref{lem:drift_lemma} therefore implies
    \[
        \sup_n \E\|Y(n)\|_1<\infty.
    \]
    Since
    \[
        \|\mu_n-p^*\|_{\operatorname{TV}}=\frac{\|Y(n)\|_1}{2n},
    \]
    the result follows.
\end{proof}

\end{document}